\documentclass[psamsfonts]{amsart}

\usepackage{amsmath, amssymb,amsfonts}
\usepackage[all,arc]{xy}
\usepackage{enumerate}
\usepackage[all]{xy}
\usepackage{mathrsfs}
\usepackage{array}
\usepackage{verbatim}

\newtheorem{thm}{Theorem}[subsection]

\newtheorem{cor}[thm]{Corollary}

\newtheorem{prop}[thm]{Proposition}
\newtheorem{lem}[thm]{Lemma}

\newtheorem{conj}[thm]{Conjecture}

\theoremstyle{definition}
\newtheorem{defn}[thm]{Definition}

\theoremstyle{remark}
\newtheorem{rem}[thm]{Remark}

\makeatletter
\let\c@equation\c@thm
\makeatother
\numberwithin{equation}{section}

\begin{document}

\title{Fields of Rationality of Automorphic Representations: the Case of Unitary Groups}

\author{John Binder}
\email{binderj@math.mit.edu}
\address{
	Department of Mathematics\\Massachusetts Institute of Technology\\77 Massachusetts Avenue\\Cambridge, MA, USA}
	

\newcommand{\NN}{\mathscr{N}}
\newcommand{\CC}{\mathbb{C}}
\newcommand{\DDD}{\mathscr{D}}
\newcommand{\HH}{\mathcal{H}}
\newcommand{\RR}{\mathbb{R}}
\newcommand{\RRR}{\mathscr{R}}
\newcommand{\FF}{\mathbb{F}}
\newcommand{\KK}{\mathscr{K}}
\newcommand{\UU}{\mathscr{U}}
\newcommand{\II}{\mathscr{I}}
\newcommand{\EE}{\mathscr{E}}
\newcommand{\GG}{\mathscr{G}}
\newcommand{\PZ}{\mathbb{P}_{\math    bb{Z}}}
\newcommand{\ZZ}{\mathbb{Z}}
\newcommand{\PP}{\mathscr{P}}
\newcommand{\PPP}{\mathscr{P}}
\newcommand{\SSS}{\mathcal{S}}
\newcommand{\LL}{\mathcal{L}}
\newcommand{\MM}{\mathscr{M}}
\newcommand{\AAA}{\mathbb{A}}
\newcommand{\GGG}{\mathscr{G}}
\newcommand{\AAAA}{\AAA}
\newcommand{\Df}{\mathcal{D}_F}
\newcommand{\QQ}{\mathbb{Q}}
\newcommand{\QQQ}{\mathscr{Q}}
\newcommand{\DD}{\mathscr{D}}
\newcommand{\OO}{\mathcal{O}}
\newcommand{\VV}{\mathscr{V}}
\newcommand{\pp}{\mathfrak{p}}
\newcommand{\qq}{\mathfrak{q}}
\newcommand{\faa}{\mathfrak{a}}
\newcommand{\mm}{\mathfrak{m}}
\newcommand{\IIII}{\mathcal{I}}
\newcommand{\JJ}{\mathscr{J}}
\newcommand{\weak}{\rightharpoonup}
\newcommand{\weaks}{\rightharpoonup^*}
\newcommand{\simga}{\sigma}
\newcommand{\linf}[2]{\left\langle {#1},\, {#2}\right\rangle}
\newcommand{\into}{\hookrightarrow}
\newcommand{\im}{\text{im}}
\newcommand{\lists}[3]{{#1}_1{#2}\ldots {#2}{#1}_{#3}}
\newcommand{\qr}[2]{\left(\frac{#1}{#2}\right)}
\newcolumntype{M}{>{$}c<{$}}
\newcommand{\alg}{\text{alg}}
\newcommand{\Spec}{\mathop{\mathrm{Spec}}\nolimits}
\newcommand{\Proj}{\mathop{\mathrm{Proj}}\nolimits}
\newcommand{\Tr}{\mathop{\mathrm{Tr}}\nolimits}
\newcommand{\Hom}{\mathop{\mathrm{Hom}}\nolimits}
\newcommand{\spa}{\mathop{\mathrm{sp}}\nolimits}
\newcommand{\rank}{\mathop{\mathrm{rank}}\nolimits}
\newcommand{\Pic}{\mathop{\mathrm{Pic}}\nolimits}
\newcommand{\image}{\mathop{\mathrm{Im}}\nolimits}
\newcommand{\tors}{\mathop{\mathrm{tors}}\nolimits}
\newcommand{\ord}{\mathop{\mathrm{ord}}\nolimits}
\newcommand{\Imag}{\mathop{\mathrm{Im}}\nolimits}
\newcommand{\trdeg}{\mathop{\mathrm{trdeg}}\nolimits}
\newcommand{\codim}{\mathop{\mathrm{codim}}\nolimits}
\newcommand{\hei}{\mathop{\mathrm{ht}}\nolimits}
\newcommand{\sgn}{\mathop{\mathrm{sgn}}\nolimits}
\newcommand{\Gal}{\mathop{\mathrm{Gal}}\nolimits}
\newcommand{\supp}{\mathop{\mathrm{supp}}\nolimits}
\newcommand{\Cl}{\mathop{\mathrm{Cl}}\nolimits}
\newcommand{\CaCl}{\mathop{\mathrm{Ca\,Cl}}\nolimits}
\newcommand{\Div}{\mathop{\mathrm{Div}}\nolimits}
\newcommand{\Sym}{\mathop{\mathrm{Sym}}\nolimits}
\newcommand{\coker}{\mathop{\mathrm{coker}}\nolimits}
\newcommand{\imag}{\mathop{\mathrm{im}}\nolimits}
\newcommand{\End}{\mathop{\mathrm{End}}\nolimits}
\newcommand{\Frob}{\mathop{\mathrm{Frob}}\nolimits}
\newcommand{\Ann}{\mathop{\mathrm{Ann}}\nolimits}
\newcommand{\Art}{\mathop{\mathrm{Art}}\nolimits}
\newcommand{\rec}{\mathop{\mathrm{rec}}\nolimits}
\newcommand{\Aut}{\mathop{\mathrm{Aut}}\nolimits}
\newcommand{\Ad}{\mathop{\mathrm{Ad}}\nolimits}
\newcommand{\nr}{\mathop{\mathrm{nr}}\nolimits}
\newcommand{\cond}{\mathop{\mathrm{cond}}\nolimits}

\newcommand{\Ind}{\mathop{\mathrm{Ind}}\nolimits}
\newcommand{\spann}{\mathop{\mathrm{span}}\nolimits}
\newcommand{\Vol}{\mathop{\mathrm{Vol}}\nolimits}
\newcommand{\Irr}{\mathop{\mathrm{Irr}}\nolimits}
\newcommand{\Res}{\mathop{\mathrm{Res}}\nolimits}
\newcommand{\genus}{\mathop{\mathrm{genus}}\nolimits}
\newcommand{\scusp}{\mathop{\mathrm{scusp}}\nolimits}
\newcommand{\PProj}{\mathbb{P}\mathop{\mathrm{roj}}\nolimits}
\newcommand{\cones}[3]{\langle v_{#1},\, v_{#2},\, v_{#3}\rangle}
\newcommand{\Rsheaf}{\mathscr{R}}
\newcommand{\Qsheaf}{\mathscr{Q}}
\newcommand{\Ksheaf}{\mathscr{K}}
\newcommand{\Hsheaf}{\mathscr{H}}
\newcommand{\Msheaf}{\mathscr{M}}
\newcommand{\Rei}{\mathcal{R}}
\newcommand{\Rie}{\Rei}
\newcommand{\hol}{\text{hol}}
\newcommand{\Nsheaf}{\mathscr{N}}
\newcommand{\unr}{\text{unr}}
\newcommand{\sHom}{\mathcal{H}om}
\newcommand{\smallmat}[4]{\left(\begin{smallmatrix}{#1} & {#2} \\ {#3} & {#4} \end{smallmatrix} \right)}
\newcommand{\twomat}[4]{\begin{pmatrix}{#1} & {#2} \\ {#3} & {#4} \end{pmatrix}}
\newcommand{\Proh}{\Proj}
\newcommand{\jj}{\mathfrak{j}}
\newcommand{\old}{\text{old}}
\newcommand{\bs}{\backslash}
\newcommand{\diam}[1]{\langle {#1} \rangle}
\newcommand{\Alg}{\textbf{Alg\,}}
\newcommand{\BB}{\mathcal{B}}
\newcommand{\Detla}{\Delta}
\newcommand{\iso}{\xrightarrow{\sim}}
\newcommand{\dep}{\mathop{\mathrm{dep}}\nolimits}
\newcommand{\ind}{\mathop{\mathrm{ind}}\nolimits}
\newcommand{\vol}{\mathop{\mathrm{vol}}\nolimits}
\newcommand{\tr}{\mathop{\mathrm{tr}}\nolimits}
\newcommand{\Stab}{\mathop{\mathrm{Stab}}\nolimits}
\newcommand{\St}{\mathop{\mathrm{St}}\nolimits}
\newcommand{\meas}{\mathop{\mathrm{meas}}\nolimits}
\newcommand{\disc}{\mathop{\mathrm{disc}}\nolimits}
\newcommand{\cusp}{\mathop{\mathrm{cusp}}\nolimits}
\newcommand{\spec}{\mathop{\mathrm{spec}}\nolimits}
\newcommand{\geom}{\mathop{\mathrm{geom}}\nolimits}
\newcommand{\pl}{\mathop{\mathrm{pl}}\nolimits}
\newcommand{\new}{\mathop{\mathrm{new}}\nolimits}

\newcommand{\Lie}{\mathop{\mathrm{Lie}}\nolimits}

\newcommand{\trace}{\mathop{\mathrm{trace}}\nolimits}
\newcommand{\mupl}{\widehat \mu^{\mathop{\mathrm{pl}}}}
\newcommand{\nupl}{\widehat \nu^{\mathop{\mathrm{pl}}}}
\newcommand{\mucusp}{\widehat \mu^{\cusp}}
\newcommand{\mudisc}{\widehat \mu^{\disc}}
\newcommand{\GL}{\mathop{\mathrm{GL}}\nolimits}
\newcommand{\PGL}{\mathop{\mathrm{PGL}}\nolimits}
\newcommand{\SL}{\mathop{\mathrm{SL}}\nolimits}
\newcommand{\EP}{\mathop{\mathrm{EP}}\nolimits}
\newcommand{\Sp}{\mathop{\mathrm{Sp}}\nolimits}
\newcommand{\ad}{\mathop{\mathrm{ad}}\nolimits}
\newcommand{\reg}{\mathop{\mathrm{reg}}\nolimits}

\newcommand{\one}{\mathbf{1}}

\newcommand{\n}{\mathfrak{n}}
\newcommand{\nn}{\mathfrak{n}}
\newcommand{\ff}{\mathfrak{f}}
\newcommand{\oo}{\mathfrak{o}}
\newcommand{\ft}{\mathfrak{t}}
\newcommand{\dd}{\mathfrak{d}}
\newcommand{\fa}{\mathfrak{a}}
\newcommand{\XX}{\mathfrak{X}}
\newcommand{\fg}{\mathfrak{g}}

\newcommand{\wh}{\widehat}
\newcommand{\Fam}{\mathcal{F}}
\newcommand{\K}{\mathbf{K}}

\begin{abstract} This paper examines fields of rationality in families of cuspidal automorphic representations of unitary groups.  Specifically, for a fixed $A$ and a sufficiently large family $\Fam$, a small proportion of representations $\pi\in \Fam$ will satisfy $[\QQ(\pi):\QQ] \leq A$.  Like earlier work of Shin and Templier, the result depends on a Plancherel equidistribution result for the local components of representations in families.  An innovation of our work is an upper bound on the number of discrete series $\GL_n(L)$ representations with small field of rationality, counted with appropriate multiplicity, which in turn depends upon an asymptotic character expansion of Murnaghan and formal degree computations of Aubert and Plymen.
\end{abstract}

\maketitle

\section{Introduction}

Let $F$ be a totally real field and $E/F$ be a totally imaginary quadratic extension. For $n \geq 2$, let $U_{E/F}(n)$ denote the quasi-split unitary group in $n$ variables corresponding to $E/F$.  It is our goal to examine the fields of rationality of cuspidal automorphic $U(n)$ representations in families.

We start with a definition, following \cite{Wal85}:
\begin{defn} Let $G$ be a topological group, let $V$ be a complex vector space and $\pi: G\to\Aut_\CC(V)$ be a smooth $G$ representation (so every $v\in V$ has an open stabilizer).

Given $\sigma \in \Aut(\CC)$, let $V_\sigma$ be a complex vector space equipped with a $\sigma$-linear isomorphism let $t_\sigma: V \to V_\sigma$ (so that $t_\sigma(cv) = \sigma(c) t_\sigma(v)$).  We define $^\sigma\pi: G \to \Aut(V_\sigma)$ via $^\sigma \pi(g) = t_\sigma \circ \pi(g) \circ t_\sigma^{-1}$.

Let
$$\Stab(\pi) = \{\sigma\in \Aut(\CC): \,^\sigma \pi \cong \pi\}.$$
The \emph{field of rationality} $\QQ(\pi)$ is the subfield of $\CC$ fixed by $\Stab(\pi)$.

If $G/F$ is a connected reductive group and $\pi$ is an automorphic $G(\AAA_F)$ representation, the field of rationality $\QQ(\pi)$ is the compositum of the fields $\QQ(\pi_\pp)$ over the finite places $\pp$ of $F$.
\end{defn}

The question of fields of rationality in families was first studied by Serre in \cite{Ser97} in the language of classical cusp forms.  Specifically, let $k \geq 2$ be an even weight, let $N \geq 1$ be a level, and define $B_k(\Gamma_0(N))$ to be the standard basis of Hecke eigenforms of weight $k$ and level $\Gamma_0(N)$ (the existence of such a `standard basis' is guaranteed by \cite{AL70}).  If $f \in B_k(\Gamma_0(N))$, then $f$ has a \emph{$q$-expansion} 
$$f(z) = q + a_2q^2 + a_3q^3 + \ldots \,\,\,\,\,\,\,\,\, \text{for } q = e^{2\pi i z}.$$

We define the \emph{field of rationality} $\QQ(f) = \QQ(a_2,\, a_3,\, \ldots)$; this is a finite extension of $\QQ$.  Serre proved:
\begin{thm} For $A \geq 1$, let
$$B_k^{\leq A}(\Gamma_0(N)) =  \big\{f\in B_k(\Gamma_0(N)) \,\big| \,[\QQ(f):\QQ] \leq A\big\}.$$

Fix an auxiliary prime $p_0$ and let $\{N_\lambda\}$ be a sequence of levels such that $(N_\lambda,\, p_0) = 1$ and $N_\lambda \to \infty$.  Then
$$\lim_{\lambda \to \infty} 
	\frac{|B_k^{\leq A}(\Gamma_0(N_\lambda))|} 
		{|B_k(\Gamma_0(N_\lambda))|}
	= 0.$$
\end{thm}

It is enlightening to reformulate the problem in terms of cuspidal automorphic representations.  There is classical construction associating to each holomorphic Hecke eigenform $f$ a cuspidal automorphic $\GL_2(\AAA)$ representation $\pi_f$.  If $f$ is a form of level $\Gamma_0(N)$, then $\pi_f$ has trivial central character and a nonzero $\Gamma_0(N)$-fixed vector; the fixed vector is unique (up to scalar multiple) if $f$ is a newform of level $\Gamma_0(N)$. The weight $k$ of 
$f$ determines the infinite component $\pi_{f,\infty}$.  Finally, it follows from strong multiplicity one that $\QQ(f) = \QQ(\pi_f)$.  

Serre posited that his theorem could be extended to arbitrary sequences of levels; this was completed by the author in \cite{Bin15}, following work of Shin and Templier in \cite{ST13}.  In their paper, Shin and Templier considered fields of rationality of representations of classical groups in families more generally.  In particular, they considered representations of classical groups in \emph{level families}.  Given a classical group $G$ over a totally real field $F$, let $\xi$ be an irreducible, algebraic, finite-dimensional $G(F_\infty) = G(F\otimes_{\QQ} \RR)$ representation.  Assume for simplicity that the highest weight of $\xi$ is regular. If $\nn$ is an ideal of $F$, let $\Gamma(\nn) \leq G(\AAA_F^\infty)$ denote the principal congruence subgroup of level $\nn$.  Let $\Fam(\xi,\, \Gamma(\nn))$ denote the family of cuspidal automorphic representations $\pi$ such that $\pi_{\infty}$ is $\xi$-cohomological and $\pi^{\infty}$ has a $\Gamma(\nn)$-fixed vector. For technical reasons the representations in $\Fam$ are counted `with multiplicity'; we invite the reader to see (\ref{multiplicity}) for the definition.  Let $\Fam^{\leq A}$ denote the subfamily consisting of those representations $\pi$ such that $[\QQ(\pi):\QQ] \leq A$.  Then Shin and Templier prove:
\begin{thm} Assume $G/F$ is an orthogonal, unitary, or symplectic group.  Let $\nn_\lambda$ be a sequence of levels with $N(\nn_\lambda) \to \infty$.  Let $\pp_0$ be a finite place of $F$ such that either:
\begin{enumerate}[(i)]
	\item $\nn_\lambda$ is coprime to $\pp_0$ for all $\lambda$, or
	\item $\ord_{\pp_0}(\nn_\lambda) \to \infty$ as $\lambda \to \infty$.
\end{enumerate}
Then $|\Fam^{\leq A}(\xi,\,\Gamma(\nn_\lambda))|/|\Fam(\xi,\,\Gamma(\nn_\lambda))| \to 0$ as $\lambda \to \infty$.
\end{thm}

It is our goal to prove an `unconditional' version of this theorem by removing the hypotheses (i) or (ii) in the case where $G$ is the unitary group $U_{E/F}(n)$.  We note one major distinction between our work and that of Shin-Templier: we will consider families $\Fam$ with fixed central character.  This distinction makes the problem more interesting: if we do not fix the central character, then the central characters $\chi_\pi$ of representations $\pi \in \Fam(\xi,\, \Gamma(\nn))$ are asymptotically equdistributed among all automorphic characters $\chi$ with conductor dividing $\nn$.  If $\nn$ is large, then many of these characters $\chi$ have large field of rationality, and one easily checks that $\QQ(\pi) \supseteq \QQ(\chi_\pi)$.  By fixing the central character, this argument fails and we instead make a deeper argument using $\pp$-adic representation theory following \cite{ST13}.

To state our theorem, let $G = U_{E/F}(n)$ for $n \geq 2$ and a CM extension $E/F$.  Also, fix the following data:
\begin{itemize}
	\item First, an irreducible, finite dimensional algebraic $G(F_\infty)$ representation $\xi$.  We assume the highest weight of $\xi$ is regular.
	\item Second, an automorphic character $\chi: Z(\AAA) \to \CC^\times$ with conductor $\ff$.  We assume $\chi_\infty$ is equal to the central character of $\xi$.
\end{itemize}

Our primary theorem is this:
\begin{thm} \label{MainThmIntro} Let $\nn$ be an ideal of $F$ divisible by $\ff$.

Given the above data, let 
$$\Fam(\xi,\,\chi,\, \Gamma(\nn))$$
denote the family of cuspidal automorphic $U(n,\,\AAA)$ representations $\pi$ such that $\pi_{\infty}$ is $\xi$-cohomological, $\chi_\pi = \chi$, and such that $\pi^{\infty}$ has a $\Gamma(\nn)$-fixed vector, counted with the appropriate multiplicity.  Let $\Fam^{\leq A}(\xi,\, \chi,\, \Gamma(\nn))$ denote the subfamily of representations $\pi$ with $[\QQ(\pi):\QQ]$.

If $\{\nn_\lambda\}$ is a sequence of ideals, divisible by $\ff$, such that $N(\nn_\lambda) \to \infty$ as $\lambda\to\infty$, then
$$\lim_{\lambda \to \infty}
	 \frac{|\Fam^{\leq A}(\xi,\,\chi,\, \Gamma(\nn_\lambda))|}
		{|\Fam(\xi,\,\chi,\, \Gamma(\nn_\lambda))|}
	= 0.$$
\end{thm}

We briefly comment on our choice of algebraic group.  Over the course of the paper, we will reduce the proof to a careful analysis of the depth-zero discrete series representations of $U(n,\, F_\pp)$ for finite places $\pp$ of $F$.  When $E$ splits at $\pp$, we have $U(n,\, F_\pp) \cong \GL_n(F_\pp)$, and the representation-theoretic properties of the depth-zero discrete series $\GL_n(F_\pp)$ representations are very well-understood following, for instance, \cite{How77}, \cite{Rod82}, \cite{Mur03}.  It seems likely that similar results should hold for other classical groups, but we are currently unable to prove them.

Moreover, we have opted to discuss unitary groups, rather than other twists of general linear groups, because of complications at the Archimedean places.  In particular, $U(n, F_\infty)$ has an anisotropic maximal torus $T$ and this greatly simplifies the analysis.  Moreover if $\xi$ is a finite-dimensional irreducible algebraic $G(F_\infty)$ representation with regular highest weight, then any automorphic representation whose infinite component is $\xi$-cohomological will have finite field of rationality.  Therefore, our theorem is not vacuous.  (Because $\GL_2(F_\infty)$ has a maximal torus which is anisotropic modulo the center, the same holds for $\GL_2$.  The reader can everywhere replace $U(n)$ with $\GL_2$ and the results go through exactly as stated.)

This paper is organized as follows: In Section 2, we discuss some introductory materials, define our families of cuspidal automorphic representations, and state our main Theorem. In section 3, we give an equidistribution theorem for the local components of families of cuspidal automorphic representations.  In particular, the local components of our families are equidistributed according to the \emph{Plancherel measure}.  In section 4, we give a proof of the main theorem contingent upon a result about discrete series representations of $\GL_n(F_\pp)$ having small field of rationality (Proposition \ref{DiscreteSeriesSmallCor}).  Because the proof necessarily involves the representation theory of $\GL_n(F_\pp)$, we have opted to delay it until Sections 5 and 6.  In Section 5, we give explicit lower bounds on the degree of the field of rationality of $\GL_n(F_\pp)$ representations of positive depth and show that the representations with small field of rationality are not `too numerous.'  In Section 6, we compute the `multiplicities' of these representations and show that they compose a small proportion of the space when counted `with multiplicity', completing the proof of Proposition \ref{DiscreteSeriesSmallCor}.  Finally, in Section 7 we present some conjectures and results that may be useful for extending our results to other classical groups.

\subsection{Acknowledgements:} I am grateful to my advisor, Sug Woo Shin, for his continued support and interest in this project. Thanks also to Julee Kim and David Vogan for many helpful discussions about the complications at the non-Archimedean and Archimedean places respectively.

\section{Introductory Materials}

\subsection{Unitary groups and maximal special subgroups}
\label{UnitaryGroups}

In this section, we give a quick primer on the quasi-split unitary group $U_{E/F}(n)$.  Let $F$ be a totally real number field and $E/F$ a totally imaginary quadratic extension; then $E$ is a CM field and the nontrivial element in $\Gal(E/F)$ acts as complex conjugation for every embedding $E \into \CC$: we denote this automorphism $x \mapsto \overline x$.

Let $\Phi = \Phi_n$ denote the matrix with entries
$$\Phi_{ij} = 
	\begin{cases} (-1)^{i-1} & i + j = n + 1
	\\ 0 &\text{otherwise}\end{cases}
$$
and let $U_{E/F}(n, R)$ be the group of $g\in \GL_n(E \otimes R)$ with $g \Phi_n \overline g^t = \Phi_n$.  This defines a connected reductive group {over F}.  The algebraic subgroup of upper-triangular matrices is a Borel subgroup, so that $U(n)$ is {quasi-split} over $F$.  Moreover, $U_{E/F}(n)$ becomes isomorphic $\GL_n$ after base-changing to $E$.  Therefore, if $v$ is any place of $F$ such that $E$ splits at $v$, then $U_{E/F}(n, F_v) \cong \GL_n(F_v)$.  In this case, we say $U(n)$ splits at $v$.

Let $\pp$ be a finite prime of $F$.  If $U(n)$ splits at $\pp$ then $\K_\pp \cong \GL_n(\oo_{F,\pp})$ is a hyperspecial maximal compact subgroup of $\GL_n(F_\pp)$.  Otherwise, $U(n,\, F_\pp)$ has a maximal hyperspecial subgroup $\K_\pp$ whenever $E/F$ is unramified.  If $E/F$ is ramified, then $U(n,\, F_\pp)$ will only have a \emph{special} maximal compact subgroup $\K_\pp$.  In either case, there is a group scheme $\mathscr{G}$ over $\oo_{F,\pp}$ whose generic fiber is isomorphic to $U_n$, and $\K_\pp = \mathscr{G}(\oo_{F,\pp})$.  If $E/F$ is unramified, then the special fiber of $\mathscr{G}$ is a connected reductive group.

We invite the reader to see \cite{Tit79} for the definition of special and hyperspecial subgroups.  In particular, special (resp. hyperspecial) subgroups are the stabilizers of special (resp. hyperspecial) \emph{points} in the Bruhat-Tits building; these are defined in 1.9 (resp 1.10).  In the non-split case, we will not give an explicit description of the maximal special and hyperspecial subgroups of $U(n)$.  Rather, we refer the reader Section 3 of \cite{GHY01}.

\begin{defn} \label{FullLevelDef} Let $\pp$ be a prime of $F$, let $\K_\pp$ be a maximal special subgroup of $U_n(F_\pp)$, and let $\mathscr{G}$ be a group scheme over $\oo_{F,\pp}$ whose generic fiber is isomorphic to $U(n)$, with $\mathscr{G}(\oo_{F,\pp}) = \K_\pp$.  For $r > 0$, we define the \emph{principal congruence subgroups} $\Gamma(\pp^r)\leq \K_\pp$ as the kernel of the canonical map $\mathscr{G}(\oo_{F,\pp}) \to \mathscr{G}(\oo_{F,\pp}/\pp^r)$.

If $U(n)$ splits at $\pp$, we will assume $\K_\pp = \GL_n(\oo_{F,\pp})$. Then $\Gamma(\pp^r)$ is subgroup $1 + \pp^r M_n(\oo_{F,\pp})$.

If $\nn = \prod_{\pp} \pp^{r_\pp}$ we set
$$\Gamma(\nn) = \left(\prod_{\pp \mid \nn} \pp^{r_\pp}\right) \times \left(\prod_{\pp \nmid \nn} K_\pp\right);$$
this is an open compact subgroup of $G(\AAA^\infty)$.
\end{defn}

Note that this definition is equivalent to the definition given in \cite{ST12} (see page 65 of that paper).

\subsection{The tempered spectrum of $\GL_n(L)$}

Throughout, let $L$ be a $\pp$-adic field with ring of integers $\oo = \oo_L$.  All representations will be assumed to be admissible and unitary.

\begin{defn} Let $(\pi,\,V)$ be an admissible irreducible $\GL_n(L)$ representation and let $(\pi^*,\, V^*)$ be its contragredient representation.  For $v\in V,\, v^*\in V^*$, we define the \emph{matrix coefficient} $f_{v,v^*}: \GL_n(L) \to \CC$ via:
$$f_{v,v^*}(g) = \linf{v^*}{\pi(g) v}.$$

\begin{enumerate}[(a)]
	\item We say $\pi$ is \emph{supercuspidal} if its matrix coefficients are compactly-supported modulo the center $Z$.
	\item We say $\pi$ is a \emph{discrete series} representation if its matrix coefficients are in $L^2(\GL_n(L)/Z)$.
	\item We say $\pi$ is \emph{tempered} if its matrix coefficients are in $L^{2 + \epsilon}(\GL_n(L)/Z)$ for every $\epsilon > 0$.
\end{enumerate}
\end{defn}

In this section, we'll briefly describe the tempered spectrum of $G = \GL_n(L)$; let $Z$ denote the center of $G$.  Let $P_0$ denote the subgroup of consisting of the upper-triangular matrices; this is a minimal parabolic subgroup.  We say $P$ is a \emph{standard parabolic} subgroup if $P\geq P_0$. In this case, the Levi component $M$ of $P$ is a standard Levi subgroup.  If $\pi_M$ is an admissible unitary $M$ representation, we define $I_M^G\pi_M$ as follows: first, let $\delta_P$ be the modulus character of $M$ acting on $P$ and let $\pi_P$ denote pullback under $P \twoheadrightarrow M$ of $\pi_M \otimes \delta_P$, and then let $I_M^G \pi' = \Ind_P^G \pi'_P$.

Let $m = nd$ and let $\pi'$ be a unitary supercuspidal representation of $\GL_m(L)$.  Let $M$ be the standard Levi subgroup of $\GL_n(L)$ isomorphic to $\GL_m(L)^d$.  Let $\pi'_M$ denote the (external) tensor product
$$\left(\pi' \otimes|\det|^{\frac{1-d}{2}}\right) \otimes \left(\pi' \otimes|\det|^{\frac{3-d}{2}}\right) \otimes \ldots \otimes \left(\pi'\otimes |\det|^{\frac{d-1}{2}}\right).$$
Then 
\begin{lem}  \begin{enumerate}[(i)]
	\item $I_M^G \pi'_M$ has a unique irreducible quotient module, which we call $\Sp(\pi',\, d)$.
	\item $\Sp(\pi',\, d)$ is a discrete series $G$ representation.
	\item All discrete series representations of $G$ are isomorphic to $\Sp(\pi',\, d)$, for some $d\mid n$ and supercuspidal representation $\pi'$ of $\GL_{n/d}(L)$.
	\item $\Sp(\pi',\, d) \cong \Sp(\pi'',\, d')$ if and only if $\pi' \cong\pi''$ and $d = d'$.
\end{enumerate}
\end{lem}
\begin{proof} (i) is Proposition 2.10 of \cite{Zel80}.  (ii) follows from \cite{BZ77}. (iii) is Proposition 11 of \cite{Rod82}. (iv) follows because $\pi'$ is the unique unitary supercuspidal representation with an unramified twist occurring in the supercuspidal support of $\Sp(\pi',\, d)$.
\end{proof}

\begin{lem}\label{tempered} Let $\omega_i$ be a discrete series $\GL_{n_i}(L)$ representation for $i =1,\ldots,\, r$, with $n_1 + \ldots + n_r = n$.  Let $M$ denote the subgroup of block diagonal matrices isomorphic to $\prod_i \GL_{n_i}(L)$.  Let $M' = \prod_j \GL_{n'_j}(L)$. Then
\begin{enumerate}[(i)]
	\item $I_{M}^G (\omega_1\otimes \ldots\otimes \omega_r)$ is irreducible and tempered,
	\item all tempered $\GL_n(L)$ representations are isomorphic to one of this form, and
	\item $I_{M}^G(\omega_1 \otimes \ldots \otimes\omega_r) \cong I_{M'}^G(\omega_1' \otimes \ldots \otimes \omega_{r'}')$ if and only if $r = r'$ and there is a permutation $s$ of $\{1,\ldots,\, r\}$ such that $\omega_i \cong \omega'_{s(i)}$.
\end{enumerate}
\end{lem}
\begin{proof} (i) and (ii) are proven in \cite{Jac77}. (iii) follows by examining the supercuspidal support of the two representations.
\end{proof}

\subsection{Euler-Poincar\'{e} functions at the Archimedean places}

Let $G/\RR$ be a reductive group.  Throughout this chapter, we will assume that $G$ has a maximal torus which is anisotropic modulo the center.  Let $A_G$ denote the maximal split torus in the center of $G\times_{\QQ} \RR$ and let $A_{G,\infty}$ denote the connected component of $A_{G}(\RR)$ (with respect to the real topology).  Let $K_\infty$ be a maximal compact subgroup of $G(\RR)$ and let $K_{\infty}' = K_\infty A_{G,\infty}$.  Fix an irreducible finite dimensional algebraic $G(\RR)$ representation $\xi$ and let $\omega_\xi$ denote the central character of $\xi$ on $A_{G,\infty}$.  Let $\pi$ be an irreducible admissible representation whose central character on $A_{G,\infty}$ is $\omega_\xi$.  Let $\fg = \Lie G(\RR)$.  The \emph{Euler-Poincar\'{e} characteristic} of $\pi$ (with respect to $\xi$) is defined as 
$$\chi_{\EP}(\pi \otimes \xi^{\vee}) = \sum_{i \geq 0} (-1)^i \dim H^i(\fg,\,K_{\infty}',\, \pi \otimes \xi^{\vee})$$
(here the cohomology is Harish-Chandra's $(\fg,\, K)$ cohomology).

We say $\pi$ is $\xi$-cohomological if there is an $i \geq 0$ such that 
$$H^i(\fg,\,K_{\infty}',\, \pi \otimes \xi^{\vee})\neq 0.$$
More generally, if $\pi$ is an automorphic $G(\AAA)$ representation such that $\pi_\infty$ is $\xi$-cohomological, we say $\pi$ is $\xi$-cohomological.  It is clear that $\chi_{\EP}(\pi_\infty \otimes \xi^{\vee}) = 0$ if $\pi$ is not $\xi$-cohomological.

\begin{defn} Let $\xi$ be an irreducible finite dimensional algebraic representation of $G(\RR)$ and let $T$ be a compact torus of $G$ of maximal dimension.  We recall that $\xi\mid_{T(\RR)}$ decomposes as a direct sum of abelian characters $\{\lambda\}$.  A choice of positive roots of $T$ determines an ordering of the roots $\{\lambda\}$, and with respect to this ordering $\xi$ has a unique highest weight $\lambda_\xi$.  We say $\xi$ has \emph{regular highest weight} if for every coroot $\alpha^{\vee}$, we have $\linf{\lambda_\xi}{\alpha^{\vee}} \neq 0$.
\end{defn}

\begin{prop} \label{RHWProps} Let the highest weight of $\xi$ be regular and let $\pi$ be an automorphic, $\xi$-cohomological representation.  Let $q(G) = \frac{1}{2} \dim_{\RR} G(\RR)/K'_\infty$.
\begin{enumerate}[(a)]
	\item If $\pi$ is $\xi$-cohomological, then $\pi_\infty$ is a discrete series representation, and $\chi_{\EP}(\pi_\infty\otimes \xi^{\vee}) = (-1)^{q(G)}.$  Moreover, all $\xi$-cohomological representations are in the same discrete series $L$-packet.
	\item $\pi$ occurs in the discrete spectrum of $G(\AAA)$ if and only if it occurs in the cuspidal spectrum, and $m_{\disc}(\pi) = m_{\cusp}(\pi)$.
	\item For any place $v$ of $F$, $\pi_v$ is tempered.
	\item The field of rationality $\QQ(\pi)$ is a finite extension of $\QQ$.
\end{enumerate}
\end{prop}
\begin{proof} (a) is the second bullet point of page 44 of \cite{ST12}.  (b) is Theorem 4.3 of \cite{Wal84}. (c) is a statement of Corollary 4.16 of \cite{ST12}.  Finally, (d) follows from Proposition 2.15 of \cite{ST13}, since a $\xi$-cohomological discrete automorphic representation is cuspidal, in view of (b).
\end{proof}

Let $G'$ be the compact inner form of $G$.  There is a unique Haar measure on $G(\RR)/Z(\RR)$ such that the induced measure on $G'(\RR)/Z'(\RR)$ has total measure $1$; we call this measure the \emph{Euler-Poincar\'{e} measure}.

In \cite{CD90}, Clozel and Delorme construct a bi-$K_\infty$-invariant function $\phi_\xi \in C^\infty(G(\RR))$ which satisfies
$$\phi_{\xi}(gz) = \chi_\xi^{-1}(z) \phi_\xi(g)\,\,\,\,\,\, g\in G(\RR),\, z\in Z(\RR)$$
and such that, for any $\pi$ with $\chi_\pi = \chi_\xi$, we have
\begin{equation} \label{CDFunction} \tr (\phi_{\xi}) = \chi_{\EP}(\pi \otimes \xi^{\vee})\end{equation}
where the trace is taken with respect to the Euler-Poincar\'{e} measure on $G(\RR)/Z(\FF)$.

Throughout the paper, we'll need the following two facts:
\begin{itemize} 
\item $\phi_\xi$ is \emph{cuspidal}; that is, its orbital integrals vanish on non-elliptic conjugacy classes in $G(\RR)$ (see, for instance, page 267 of \cite{Art89}).
\item  $\phi_\xi(1) = \dim\xi$; this is implicit in \cite{Art89} and follows because $\dim \xi$ is the Plancherel measure of the $L$-packet of discrete-series representations which are $\xi$-cohomological.
\end{itemize}

\subsection{Families of cuspidal automorphic representations}

Fix the following data:
\begin{itemize}
	\item A totally real number field $F$ and a totally imaginary quadratic extension $E/F$;
	\item a finite-dimensional irreducible algebraic representation $\xi$ of $U_{E/F}(n,\, F_\infty)$.  We will assume the highest weight of $\xi$ is regular;
	\item an automorphic character $\chi: Z(\AAA) \to \CC^\times$, such that $\chi_\infty = \chi_\xi$ with conductor $\ff$; and
	\item an ideal $\nn$ of $F$ that is divisible by $\ff$.
\end{itemize}

Let $\Fam(\xi,\,\chi,\,\Gamma(\nn))$ denote the multiset of cuspidal automorphic representations $\pi$ such that $\chi_\pi = \chi$ and $\pi_\infty$ is $\xi$-cohomological.  Such a representation $\pi$ is counted with multiplicity
\begin{equation} \label{multiplicity} 
a_{\Fam}(\pi) = m_{\cusp}(\pi) \cdot \dim (\pi^{\infty})^{\Gamma(\nn)}.
\end{equation}

Since the highest weight of $\xi$ is regular, we may replace $m_{\cusp}(\pi)$ by $m_{\disc}(\pi)$.

  Given a family $\Fam$, we define $\Fam^{\leq A}$ as the multiset with
$$a_{\Fam^{\leq A}}(\pi) = 
	\begin{cases} 
		a_{\Fam}(\pi) & [\QQ(\pi):\QQ] \leq A
		\\ 0 & \text{otherwise.}
	\end{cases}
$$

For a given multiset $\Fam$, there are finitely many $\pi$ such that $a_{\Fam}(\pi) \neq 0$ by a result of Harish-Chandra. We let 
$$|\Fam| = \sum_{\pi} a_{\Fam}(\pi).$$

Recall the statement of our main Theorem \ref{MainThmIntro}:
\begin{thm} \label{MainThm} Fix $E/F,\,\xi$, and $\chi$ as above.  If $\{\nn_\lambda\}$ is \emph{any} sequence of ideas divisible by the conductor $\ff$ of $\chi$ such that $N(\nn_\lambda) \to \infty$ then
$$\lim_{\lambda \to \infty} 
	\frac{|\Fam(\xi,\,\chi,\, \Gamma(\nn_\lambda))|}
		{\Fam^{\leq A}(\xi,\,\chi,\, \Gamma(\nn_\lambda))|}
= 0.$$
\end{thm}
		
\section{Plancherel equidistribution for local components of automorphic representations}

\subsection{Hecke algebras and Plancherel measure}

Throughout, let $L$ denote a $\pp$-adic field.  Let $G/L$ be a connected reductive group with center $Z$ and let $\chi:Z(L) \to \CC^\times$ be a unitary character.  We define $\Pi(G(L))$ to be the set of irreducible, admissible, unitary  $G(L)$ representations and $\Pi(G(L),\,\chi)$ is the subset consisting of those representations $\pi$ with $\chi_\pi = \chi$.  Moreover, $\Pi^t$ (resp. $\Pi^{ds}$) denote the subsets of $\Pi$ consisting of tempered (resp. discrete series) representations.

\begin{defn} We define the \emph{Hecke algebra} $\HH(G(L))$ as the convolution algebra of locally constant, compact supported functions $G(L) \to \CC$.

If $\phi\in \HH(G(L))$ and $\pi$ is an irreducible, admissible $G(L)$ representation, then the map
$$\pi(\phi): v \mapsto \int_{G(L)} \phi(g)\, \pi(g)\cdot v \,dg$$
is well-defined and of trace class.  We define $\wh \phi(\pi) = \tr \pi(\phi)$.  The map $\phi \mapsto \wh \phi$ is a linear map from $\HH(G(L))$ to the space of bounded, continuous functions on $\Pi(G(L))$ that are supported on a finite number of Bernstein components.

We define the \emph{fixed central character Hecke algebra} $\HH(G(L),\,\chi)$ as the convolution algebra of locally constant functions $\phi: G(L) \to \CC$ such that
\begin{itemize}
	\item $\phi$ is compactly supported modulo $Z(L)$, and
	\item for $g\in G(L),\, z\in Z(L)$, we have $\phi(gz) = \chi^{-1}(z)\phi(g)$
\end{itemize}

If $\phi_\chi\in \HH(G(L),\chi)$ and $\pi$ is an irreducible, admissible $G(L)$ representation with central character $\chi$, the map
$$\pi(\phi_\chi): v \mapsto \int_{G(L)/Z(L)} \phi(g)\,\pi(g)\cdot v\,dg$$
is well-defined and of trace class: we define $\wh \phi_\chi(\pi) = \tr \pi(\phi_\chi)$.  As above, this gives a linear map from $\HH(G(L),\,\chi)$ to the space of functions on $\Pi(G(L),\,\chi)$.

There is an \emph{averaging map} $\HH(G(L)) \to \HH(G(L),\, \chi)$ given by $\phi\mapsto \phi_\chi$, where 
$$\phi_\chi(g) = \int_{Z(L)} \phi(gz)\chi(z)\,dz.$$
\end{defn}

We have stated the above definition for $G(L)$ but will often apply the notation more generally.  Specifically, if $F$ is a number field and $G/F$ a connected reductive algebraic group, we may refer to the Hecke algebras $\HH(G(\AAA^\infty_F))$ and $\HH(G(\AAA^\infty_F),\,\chi)$ for a central character $\chi$.  If $S$ is a finite set of finite places of $F$ we may moreover replace $\AAA_F$ by $F_S = \prod_{\pp\in S} F_\pp$ or $\AAA^{\infty, S}$.

The following lemma is a simple application of Fubini's theorem, but will come up often in the following chapters:
\begin{lem} \label{AveragingLemma} Assume Haar measures on $Z(L),\, G(L),\, G(L)/Z(L)$ are chosen compatibly.  Fix $\phi\in \HH(G(L))$ and let $\phi_{\chi} \in \HH(G(L),\, \chi)$ be its image under the averaging map.  If $\pi \in \Pi(G(L),\,\chi)$, then $\wh \phi(\pi) = \wh \phi_{\chi}(\pi)$.  (Here $\wh \phi$ is a function on $\HH(G(L))$ and $\wh \phi_\chi$ is a function on $\HH(G(L),\,\chi)$).
\end{lem}

If $\Gamma \leq G(L)$ is an open compact subgroup, let $e_\Gamma = \vol(\Gamma)^{-1} \one_{\Gamma}$.  This is an idempotent in $\HH(G(L))$ (that is, $e_{\Gamma} \star e_{\Gamma} = e_{\Gamma}$).  Moreover, if $\pi$ is an irreducible admissible $G(L)$-representation, then 
$$\wh e_{\Gamma}(\pi) = \tr \pi(e_\Gamma) = \dim \pi^{\Gamma},$$
where $\pi^{\Gamma}$ denotes the space of $\Gamma$-fixed vectors in the space of $\pi$.  Moreover, let $e_{\Gamma,\chi}$ denote the image of $e_{\Gamma}$ under the averaging map $\HH(G(L)) \to \HH(G(L),\, \chi)$.  We note that $e_{\Gamma,\chi} = 0$ unless $\chi$ is trivial on $\Gamma \cap Z$, and in this case, $e_{\Gamma,\chi}(1) = \vol(\Gamma Z/Z)^{-1}$.  Moreover, it follows immediately that if $\pi$ has central character $\chi$, then 
$$\wh e_{\Gamma,\chi}(\pi) = \dim \pi^{\Gamma}.$$

\begin{prop} There is a unique measure $\mupl$ on $\Pi(G(L))$, called the \emph{Plancherel measure}, such that, for any $\phi \in \HH(G(L))$ the following equality holds:
$$\phi(1) = \mupl(\phi) := \int_{\Pi(G(L))} \wh \phi(\pi)\,d\mupl(\pi).$$

Moreover, $\mupl$ is supported on the tempered spectrum $\Pi^t(G(L))$.
\end{prop}

For $\pp$-adic groups, the Plancherel measure was described in \cite{Wal03}.  In the case of $G = \GL_n$, a completely explicit description of the Plancherel measure is given in \cite{AP05}.  We will need a fixed-central-character version of the Plancherel measure:

\begin{prop} There is a unique measure $\mupl_\chi$ on $\Pi(G(L),\,\chi)$ such that, for any $\phi_{\chi} \in \HH(G(L),\,\chi)$ the following equality holds:
$$\phi_{\chi}(1) = \mupl_{\chi}(\phi_\chi) := \int_{\Pi(G(L),\,\chi)} \wh \phi_{\chi}(\pi)\, d\mupl_{\chi}(\pi).$$

We call $\mupl_{\chi}$ the \emph{fixed central character Plancherel measure}; it is supported on the tempered spectrum $\Pi^t(G(L),\,\chi)$.  For any $\pi$ which is not a discrete series representation, we have $\mupl_{\chi}(\pi) = 0$.  If $\pi$ is a discrete series representation, then $\mupl_{\chi}(\pi) = \deg(\pi)$, the formal degree of $\pi$.
\end{prop}

To our knowledge, the construction of the fixed central character Plancherel measure has not been written down explicitly.  However, the construction follows from abelian Fourier analysis and the non-fixed central character Plancherel measure as in \cite{Bin15}.

\subsection{Counting measures and Plancherel Equidistribution}

For this subsection, we place ourselves in the global setting.  To this end, we fix a totally real number field $F$ and a totally imaginary quadratic extension $E/F$.  Let $G = U_{E/F}(n)$ with center $Z \cong \Res_{E/F} \mathbb{G}_m$.   Let $\AAA = \AAA_F$ denote the ad\`{e}le ring of $F$.  We fix moreover the following data:
\begin{itemize}
	\item A finite set $S$ of finite places;
	\item an irreducible, finite-dimensional, algebraic representation $\xi$ of $G(F_\infty)$ (as before, we assume the highest weight of $\xi$ is regular);
	\item an automorphic character $\chi: Z(F)\bs Z(\AAA) \to \CC^\times$ with $\chi|_{Z(F_\infty)} = \chi_{\xi}$, the central character of $\xi$; and
	\item an open compact subgroup $\Gamma \leq G(\AAA^S)$ such that $\chi$ is trivial on $\Gamma \cap Z(\AAA^S)$.
\end{itemize}

Let $\Pi_{\disc}(G,\, \chi)$ (resp. $\Pi_{\cusp}(G,\,\chi)$) denote the set of discrete (resp. cuspidal) automorphic $G(\AAA)$ representations with central character $\chi$.

\begin{defn} \label{CountingDef} Fix the above data, and define the \emph{counting measure} $\wh \mu_{\Gamma,\xi,\chi}$ on $\Pi(G(F_S),\,\chi_S)$ (with respect to $\xi$) as follows: for a subset $B\subseteq \Pi(G(F_S),\,\chi_S)$, set
\begin{align*}
\wh \mu^{\disc}_{\Gamma,\xi,\chi}(B) = &  
	\frac{(-1)^{q(G)}\cdot \vol(\Gamma Z/Z)}
		{\dim(\xi)\cdot \vol(G(\QQ)Z(\AAA)\bs G(\AAA))}
	\\ & \,\,\,\,\, \cdot\sum_{\pi \in \Pi_{\disc}(G,\, \chi)}
		m_{\disc}(\pi) 
		\cdot \chi_{\EP}(\pi \otimes \xi^{\vee}) 
		\cdot \dim (\pi^{S,\infty})^{\Gamma}
		\cdot \one_B(\pi_S).
\end{align*}

Because the highest weight of $\xi$ is regular, $m_{\disc}(\pi)$ may be replaced by $m_{\cusp}(\pi)$.  Also, in this case, if $\pi_\infty$ is $\xi$-cohomological then $\chi_{\EP}(\pi_\infty \otimes \xi) = (-1)^{q(G)}$ so the definition simplifies as
\begin{align*}
\wh \mu^{\disc}_{\Gamma,\xi,\chi}(B) = &  
	\frac{\vol(\Gamma Z/Z)}
		{\dim(\xi)\cdot \vol(G(\QQ)Z(\AAA)\bs G(\AAA))}
	\\ & \,\,\,\,\, \cdot\sum_{\substack{
			\pi \in \Pi_{\disc}(G,\, \chi)
			\\ \pi_\infty \text{ is } \xi\text{-cohomological}}}
		m_{\disc}(\pi) 
		\cdot \dim (\pi^{S,\infty})^{\Gamma}
		\cdot \one_B(\pi_S).
\end{align*}
\end{defn}

\begin{defn}\label{PlancherelEquiDefn} Let $F,\, G,\, S,\, \xi,\, \chi$ be as above, and let $\{\Gamma_{\lambda}\}_{\lambda \geq 0}$ be a sequence of open compact subgroups of $G(\AAA^S)$.  We say that $\{\Gamma_\lambda\}$ satisfies \emph{Plancherel equidistribution} with respect to $\xi$ and $\chi$ if the following hold:
\begin{itemize} 
	\item Whenever $A$ is a bounded subset of $\Pi(G(F_S),\,\chi_S)$ that does not intersect the tempered spectrum $\Pi^{t}(G(F_S),\,\chi_S)$, we have
		$$\lim_{\lambda \to \infty} \wh \mu_{\Gamma_\lambda}(A) = 0.$$
	\item Whenever $A$ is a Jordan-measurable subset of $\Pi_t(G(F_S),\,\chi_S)$, we have
		$$\lim_{\lambda \to \infty} \wh \mu_{\Gamma_\lambda}(A) = \mupl_{\chi_S}(A).$$
\end{itemize}\end{defn}
	
\begin{rem} \label{ComparisonRem} Fix $\xi,\,\chi$ and a finite set $S$ of finite places.  Let $\nn = \nn_S \cdot \nn^S$, (where $\nn_S$ is only divisible by primes in $S$ and $\nn^S$ is coprime to $S$).  If $\Fam = \Fam_{\cusp}(\xi,\,\chi,\, \Gamma(\nn_\lambda))$ and $B$ is a subset of $\Pi(G(F_S),\,\chi_S)$, then 
$$\sum_{\pi_S \in B} a_{\Fam}(\pi) = \dim(\xi)\cdot \vol(\Gamma(\nn^S)Z/Z) \cdot \mu^{\cusp}_{\xi,\,\chi}(\one_B \cdot \wh e_{\Gamma(\nn),\,\chi_S})$$
(here $e_{\Gamma(\nn_S)}$ is the idempotent corresponding to $\Gamma(\nn_S) \leq G(F_S)$ and $e_{\Gamma(\nn_S),\,\chi_S}\in \HH(G(F_S),\,\chi_S)$ is its image under the averaging map).

Therefore, the content of Plancherel equidistribution is that the local components of $\pi\in \Fam$ are equidistributed with respect to the Plancherel measure in an appropriate sense.  We will make use of this explicitly in the next section, when we take $B$ to be the set of representations with small field of rationality.
\end{rem}
It is worth noting some discrepancies between our notation (which follows that of \cite{Shi12} and \cite{ST12}), and that of \cite{FL13}, \cite{FLM14}, \cite{FL15}.  In the latter three papers, Finis, Lapid, and Mueller consider the \emph{limit multiplicity problem}, which differs from our definition in one important respect: instead of fixing an algebraic representation at $\infty$, they allow $S$ to contain the infinite places and as such consider a more general set of representations at $\infty$.  This makes their work more general, but forces them to work with more difficult versions of the trace formula, and the asymptotic vanishing of these trace formula terms is still unproven and depends on some analytic prerequisites.  In contrast, in the formulation of Shin and Templier, we may use the trace formula on test functions whose infinite components are Euler-Poincare functions (see \ref{CDFunction}), which considerably simplifies the formula. Moreover, we are ultimately interested in fields of rationality and an important source of discrete automorphic representations with finite fields of rationality are those that are cohomological with respect to certain algebraic representations.

\begin{rem} \label{HaarChoice} The definitions of the Plancherel measure and our counting measures both depend on a choice of Haar measure on the group $G(\AAA)/Z(\AAA)$; in the statement of Plancherel equidistribution we assume that we have made the same choice of Haar measure on each side.  Throughout the remainder of the paper, we will make the following choice of Haar measure.  At the infinite places, we use the Euler-Poincar\'{e} measure.  In the case where $G = U(n)$ or $\GL_n$, we'll pick Haar measures on $G(F_\pp)$ so that $\K_\pp Z/Z$ has measure $1$, where $\K_\pp$ is as defined in Subsection \ref{UnitaryGroups}.
\end{rem}

\subsection{Proof of Plancherel equidistribution}

In this section, we prove the following:
\begin{prop} Fix $\xi$ and $\chi$ and as in the previous section, where $\chi$ has conductor $\ff = \ff^S\cdot \ff_S$. Let $S$ be a finite set of finite places of $F$.  Let $\nn_{\lambda}$ be a sequence of ideals divisible by $\ff^S$ and coprime to $S$.   Then the sequence of measures $\{\mu^{\cusp}_{\Gamma(\nn_\lambda), \xi,\chi}\}$ satisfies Plancherel equidistribution.
\label{PlancherelTheorem}
\end{prop}

An analogous result (without fixing the central character) is given by Corollary 9.22 of \cite{ST12}.  Because our method of proof is basically the same as theirs, we will give a sketch, highlighting where we use fixed-central-character analogs of their steps.

The first necessary ingredient is a density theorem of Sauvageot, adapted to the fixed-central-character setting:
\begin{prop} Let $\wh f_S$ be a function on $\Pi(G(F_S),\, \chi_S)$ that is bounded, has bounded support, and is continuous outside a set of Plancherel measure zero.  Fix $\epsilon > 0$.  Then there are functions $\phi_S,\, \psi_S \in \HH(G(F_S),\, \chi_S)$ such that
\begin{enumerate}[(a)]
	\item for all $\pi\in \Pi(G(F_S),\,\chi_S)$, we have $|\wh f_S - \wh \phi_S| \leq \wh \psi_S$, and
	\item $\mupl_{\chi_S}(\wh \psi_S) = \psi_S(1) < \epsilon$.
\end{enumerate}
\end{prop}
\begin{proof} This follows from the non-fixed version of the density theorem (see \cite{Sau97}, Th\'{e}or\`{e}me 7.1) exactly as in Lemma 11.2.7 of \cite{Bin15}.
\end{proof}

\begin{cor} Assume for any function $\phi_S \in \HH(G(F_S),\,\chi_S)$ we have
$$\lim_{\lambda\to\infty} \wh \mu_{\Gamma(\nn_\lambda),\chi_S}(\wh \phi_S)  = \mupl_{\chi_S}(\wh \phi_S).$$
Then Plancherel equidisitribution holds for the sequence of measures $\{\wh \mu_{\Gamma(\nn_\lambda),\chi_S}\}$.
\end{cor}
\begin{proof} Follows from the previous proposition, as in Section 2 of \cite{FLM14}.\end{proof}

With this in hand, we can prove Plancherel equidistribution by plugging an appropriate function into the (fixed-central-character) trace formula.  Following \cite{Art89}, the trace formula simplifies considerably when applied to a test function $\phi = \phi^{\infty} \cdot \phi_\infty$ where $\phi_\infty$ is an Euler-Poincar\'{e} function.  Therefore, for the rest of the subsection we will assume that $\phi = \phi^\infty\cdot \phi_\xi\in \HH(G(\AAA),\,\chi)$, where $\phi_\xi$ is the Euler-Poincar\'{e} function corresponding to a $G(F_\infty)$ representation of regular highest weight.  We state the formula here:

\begin{defn} Let $G/F$ be a connected reductive group.  At each finite place, let $\K_{\pp}$ denote a special maximal compact subgroup (that is hyperspecial at all places where $G$ is unramified; see \ref{HaarChoice}), and let $\K^\infty = \prod_{\pp} \K_\pp$.  Let $P$ be a parabolic subgroup with Levi decomposition $P = MN$ and let $\gamma\in M(F)$.  If $\phi^\infty: G(\AAA^\infty) \to \CC$ is locally constant and compactly-supported modulo the center, define the \emph{constant term}
$$\phi^\infty_M(\gamma) = \int_{\K^\infty} \int_{N(\AAA^\infty)} \phi^\infty(k^{-1} \gamma n k)\, dn \, dk.$$

Moreover, if $\phi_M^\infty: M(\AAA^\infty) \to \CC^\times$ is locally constant and compactly supported modulo the center, and $\gamma \in M(\AAA^\infty)$, let $M_\gamma$ denote the identity component of the centralizer of $\gamma$ in $M$.  Define the \emph{orbital integral}
$$O_\gamma(\phi^\infty_M) = \int_{M_\gamma(\AAA^\infty)\bs M(\AAA^\infty)} \phi_M^\infty(m^{-1}\gamma m)\, dm.$$
\end{defn}

\begin{defn} Let $\phi = \phi^\infty \phi_\xi$, and let $\chi:Z(\AAA) \to \CC^\times$ be such that $\chi_\infty = \chi_\xi$.  Given $\gamma,\, \gamma'\in M(F)$, we say $\gamma \sim \gamma'$ if there is $m\in M(F),\, z\in Z(F)$ such that $m\gamma m^{-1}  = z \gamma'$.  The \emph{geometric expansion} of the trace formula is 
$$I_{\geom}(\phi^\infty\cdot\phi_\xi,\, \chi) = \sum_{M \geq M_0} \sum_{\gamma \in M(F)/\sim} C(M,\,\xi,\,\gamma) \cdot O_{\gamma}(\phi_M^\infty).$$

Here the outer sum runs over the set of cuspidal Levi subgroups containing a fixed minimal Levi subgroup $M_0$.  And the inner sum runs over representatives of equivalence classes of semisimple elements of $M(F)$.  The term $C(M,\,\xi,\,\gamma)$ is a constant independent of $\phi$, with $C(G,\, \xi,\, 1_G) = (-1)^{q(G)}\dim(\xi)\vol(Z(\AAA)G(F)\bs G(\AAA))$.  The exact values of the other constants $C(M,\, \xi,\, \gamma)$ are unnecessary for our purposes: we invite the reader to see the explanation after (4.3) of \cite{Shi12}.

The \emph{spectral expansion} $I_{\spec}(\phi^\infty\cdot\phi_\xi,\, \chi)$ is 
$$
\sum_{\chi_{\pi} = \chi} m_{\disc}(\pi) \cdot \tr \phi_{\xi}(\pi_\infty) \cdot \wh \phi^\infty(\pi^\infty) $$
here the sum runs over the set of discrete automorphic representations $\pi$ with central character $\chi$.
\end{defn}

Since the highest weight of $\xi$ is regular, $\tr \phi_{\xi}(\pi_{\infty})$ is zero unless $\pi_\infty$ is a discrete-series representation that is $\xi$-cohomological; in this case, $\tr \phi_{\xi}(\pi_{\infty}) = (-1)^{q(G)}$.  Therefore, the spectral expansion simplifies as
$$(-1)^{q(G)} \sum_{
		\substack{
			\chi_{\pi} = \chi
			\\ \text{$\pi$ is $\xi$-cohomological} }}
	m_{\disc}(\pi) \cdot \wh \phi^\infty(\pi^\infty)$$

Moreover, in this situation, all $\xi$-cohomological discrete automorphic representations are cuspidal, so we may replace the $m_{\disc}$ with $m_{\cusp}$ in the definition of $I_{\spec}$.

\begin{thm}[\cite{Art89}, Theorem 6.1] If $\phi_\xi$ is the Euler-Poincar\'{e} function correpsonding to $\xi$ and $\phi = \phi^{\infty} \phi_\xi \in \HH(G(\AAA),\, \chi)$, then
$$I_{\geom}(\phi^\infty \phi_\xi,\, \chi) = I_{\spec}(\phi^\infty\phi_\xi,\, \chi).$$
\label{TraceFormula}\end{thm}
\begin{proof} The non-fixed central character version of this is precisely Theorem 6.1 of \cite{Art89}.  The fixed-central-character version can be derived using abelian Fourier analysis.\end{proof}

We now complete the proof of \ref{PlancherelTheorem}.  Following \cite{ST12}, let $\phi_{\xi}$ be the Euler-Poincar\'{e} function corresponding to $\phi$.  Let $\phi_S \in \HH(G(F_S),\,\chi_S)$ be arbitrary.  Let $e_{\Gamma(\nn)}$ be the idempotent in $\HH(G(\AAA^{S,\infty}))$ corresponding to the subgroup $\Gamma(\nn)$, and let $\phi_{\nn}^{S,\infty}$ be its image in $\HH(G(\AAA^{S,\infty}), \, \chi^{S,\infty})$.  Plugging the test function
$$\phi_{\nn} =
	(-1)^{q(G)} 
	\frac{\vol(\Gamma(\nn)\cdot Z/Z)}
		{\dim \xi \cdot \vol(G(F)Z(\AAA)\bs G(\AAA))}
	\cdot \phi_S \phi_{\nn}^{S,\infty} \phi_\xi
$$
into the fixed central character trace formula, we see that the spectral side is equal to $\wh \mu^{\disc}_{\chi_S}(\wh \phi_S)$.  On the geometric side, the term corresponding to $M = G$, $\gamma = 1$ is precisely $\mupl_{\chi_S}(\phi_S)$.  The other terms equal to orbital integrals of constant terms of the functions $\phi_\nn$.  On $G(\AAA^{S,\infty})$ these functions are supported on $Z \cdot \Gamma(\nn)$, and they have absolute value $1$ on these subgroups.

By applying Lemma 8.4 of \cite{ST12} to $G^{\ad}$, we get the following result: for any element $\gamma \in G(F)$ that is not of the form $zu$ for $z\in Z$ and $u$ unipotent, if $N(\nn)$ is high enough then $\gamma$ does not conjugate into $Z \cdot \Gamma(\nn)$.  Therefore, for any Levi subgroup $M$ and $\gamma\not\in Z$, the trace formula term corresponding to $(M,\, \gamma)$ eventually vanishes.  Moreover, if $M \neq G$, then the trace formula term corresponding to $(M,\, 1_{G})$ approaches $0$ as $N(\nn_\lambda)\to\infty$ by the same logic as in the proof of Theorem 9.16 of \cite{ST12}.

Therefore, as $N(\nn)\to\infty$, the limit of the geometric side of the trace formula is $\phi_S(1) = \mupl_{\chi_S}(\wh \phi_S)$.  Moreover, one checks that the spectral side of the trace formula is precisely the counting measure $\wh\mu^{\disc}_{\Gamma(\nn),\,\xi,\,\chi_S}(\wh \phi_S)$, which is equal to the cuspidal counting measure since $\xi$ has regular highest weight. This completes the proof.

\section{Contingent Completion of the Proof}

In this section, we will prove our main theorem, contingent upon some necessary results about depth-zero discrete series representations.  Because these results require delving deeper into the representation theory of $\GL_n(L)$, we have opted to delay their proofs until the next section.

Throughout this section, we are in the global situation: $F$ is a totally real field, $E/F$ a totally imaginary quadratic extension, and $G = U_{E/F}(n)$.


\subsection{A finiteness result for local components of cohomological representations}
Let $\xi$ be an irreducible, finite dimensional, algebraic representation of $G(F_\infty)$.  We continue to assume the highest weight of $\xi$ is regular.
\begin{defn} Fix a finite place $\pp$ of $F$.  We say a representation $\pi_\pp$ of $G(F_\pp)$ is \emph{potentially $\xi$-cohomological} if there is a $\xi$-cohomological automorphic representation $\pi$ whose $\pp$-component is isomorphic to $\pi_\pp$.

Given $A \geq 1$, let $\mathcal{Z}_{\pp}(A,\, \xi)$ denote the set of potentially $\xi$-cohomological $G(F_\pp)$ representations $\pi_\pp$ with $[\QQ(\pi_\pp):\QQ] \leq A$.  (We will drop the references to $\pp,\, A,\, \xi$ when they are clear from context).
\end{defn}

\begin{prop} \label{Finiteness} Fix $\pp,\, A,\, \xi$ as above and assume the highest weight of $\xi$ is regular.  The set $\mathcal{Z}_\pp(A,\, \xi)$ is finite.\end{prop}

\begin{proof} This is Corollary 5.7 of \cite{ST13}.\end{proof}

\begin{prop} \label{LargePPower} Fix $A \in \ZZ_{\geq 1}$, $\epsilon > 0$, and a finite prime $\pp$ of $F$.  There is an $r_0 \in \ZZ_{\geq 1}$ (depending on $A,\, \epsilon,\, \pp$) such that, for any ideal $\nn$ of $F$ with $\pp^{r_0} \mid \nn$, we have
$$\frac{|\Fam^{\leq A}(\xi,\,\chi,\, \Gamma(\nn))|}
	{|\Fam(\xi,\,\chi,\, \Gamma(\nn))|} < 
\epsilon.$$
\end{prop}

\begin{rem} The reader should compare the first statement to the statement of Theorem 6.1 (ii) of \cite{ST13}.  We believe there to be a small mistake in the proof contained in that paper: namely, the authors forget to count representations $\pi$ with appropriate multiplicity. Fortunately, the proof is correct in spirit and the only missing step is to note a bound on the growth of $\dim \pi_\pp^{\Gamma(\pp^r)}$; the proof below should correct this minor oversight.
\end{rem}

\begin{proof} Write $\Fam_{\nn} = \Fam(\xi,\,\chi,\,\Gamma(\nn))$.  By Corollary \ref{Finiteness} there is a finite set $\mathcal{Z}$ (depending on $\xi$ and $A$) of potentially $\xi$-cohomological representations $\pi_\pp$ with $[\QQ(\pi_\pp):\QQ] \leq A$.  By Harish-Chandra's local character expansion, for each $\pi_\pp$, there are constants $C_{\pi_\pp}$ and $d_{\pi_\pp}$ such that $\dim \pi_\pp^{\Gamma(\pp^r)} \sim C_{\pi_\pp} q^{d_{\pi_\pp}}$.

Let $C_\xi = \vol(G(F)Z(\AAA)\bs G(\AAA)) \dim(\xi)$.  By Proposition \ref{PlancherelTheorem} applied to $S = \emptyset$, we have $|\Fam_\nn| \sim \frac{C_\xi}{\vol(\Gamma(\nn)Z/Z)}$.  

Fix a very small $\delta > 0$ and let $r$ be large enough that
\begin{enumerate}[(a)]
	\item every $\pi_\pp\in \mathcal{Z}_\pp$ has a $\Gamma(\pp^r)$-fixed vector,
	\item for $r' \geq r$ and $\pi_\pp \in \mathcal{Z}$, we have 
	$$(1 - \delta)C_{\pi_\pp} q^{d_{\pi_\pp}r}
		< \dim \pi_{\pp}^{\Gamma(\pp^r)} < 
		 (1 + \delta)C_{\pi_\pp} q^{d_{\pi_\pp}r},$$
	and
	\item for any $\nn$ with $\pp^r \mid \nn$, we have
	$$C_\xi(1 - \delta) < \vol (\Gamma(\nn)Z/Z) \cdot |\Fam_\nn| < C_\xi(1 + \delta).$$
\end{enumerate}

Now let $\nn = \pp^{r'} \cdot \ft$, with $\ft$ coprime to $\pp$ and $r' \geq r$. Let $\nn' = \pp^r \ft$, and let $\Fam_{\nn'} = \Fam(\xi,\,\chi,\,\Gamma(\nn'))$.  Assume $\pi \in \Fam_\nn$ with $\pi_\pp \in \mathcal{Z}$.  Then $\pi_\pp$ has a $\Gamma(\pp^r)$-fixed vector (by (a)) and so $\pi$ occurs in $\Fam_{\nn'}$.  For such a $\pi$ we have
$$a_{\nn}(\pi) < (1 + 3\delta)\cdot q^{d_{\pi_\pp}(r' - r)}\cdot a_{\nn'}(\pi) \,\,\,\,\,\,\,\,\, \text{by (b)}.$$
(Here $a_\nn,\, a_{\nn'}$ denote the multiplicities in $\Fam_{\nn},\, \Fam_{\nn'}$ respectively).

Moreover, by (c) we have
$$\sum_{\pi_\pp\in \mathcal{Z}} a_{\nn'}(\pi) < |\Fam_{\nn'}| \leq (1 + \delta) \frac{C_\xi}{\vol(\Gamma(\nn')Z/Z)}$$
so that
$$\sum_{\pi_\pp\in \mathcal{Z}} a_\nn(\pi) < (1 + 5\delta) \frac{C_{\xi}}{\vol(\Gamma(\nn')Z/Z)} q^{d_{\pi_\pp}(r' - r)}.$$

We note here that $d_{\pi_\pp}$ is bounded above by some $d$, the maximal dimension of a nilpotent orbit in $\Lie(G)$, so that we have
$$\sum_{\pi_\pp\in \mathcal{Z}} a_\nn(\pi) < (1 + 5\delta) \frac{C_{\xi}}{\vol(\Gamma(\nn')Z/Z)} q^{d(r' - r)}.$$
Moreover $|\Fam_\nn|$ is bounded above by $(1 - \delta)\frac{C_\xi}{\vol(\Gamma(\nn)Z/Z)}$.  Therefore, we have 
$$\frac{\left(\sum_{\pi_\pp\in \mathcal{Z}} a_\nn(\pi)\right)}
	{|\Fam_\nn|}
	< (1 + 7\delta) \frac{\vol(\Gamma(\nn)Z/Z)}{\vol(\Gamma(\nn') Z/Z)} q^{d(r' - r)}
	= (1 + 7\delta) q^{(d - \dim(G/Z))(r' - r)}$$
which can be made arbitrarily small if $r'$ is large enough, since $d < \dim(G/Z)$ for $G$ which are reductive and non-abelian.
\end{proof}

\subsection{Contingent completion of the proof}

Throughout this section, we will assume the following:

\begin{prop} \label{DiscreteSeriesSmallCor} Fix $\epsilon > 0$ and $A \in \ZZ_{\geq 1}$.  There is a $P_0\in \ZZ_{\geq 1}$ such that, for a $\pp$-adic field $L$ with residue characteristic $p > P_0$ and any unramified character $\chi: L^\times \to \CC^\times$, the following holds:

Then for any $r$
\begin{equation} 
	\sum_{\substack{
	\pi \in \Pi^{ds}(\GL_n(L),\chi)
	\\ [\QQ(\pi):\QQ] \leq A}}
		\deg(\pi)\cdot \dim \pi^{\Gamma}
\leq \epsilon \cdot \vol(\Gamma Z/Z)^{-1}.
\end{equation}
\end{prop}

Because the proof of the proposition involves delving more deeply into the representation theory of $\GL_n(L)$, we have opted to prove it in the following chapters.  In this section, we will complete the proof of our main theorems, \emph{contingent upon the proposition}.

\begin{proof}[Proof of Theorem \ref{MainThmIntro}]
Fix $A \geq 1$ and $\epsilon > 0$.  We can pick a prime $\pp$ such that
\begin{enumerate}[(a)]
	\item $G$ splits at $\pp$,
	\item $\chi_\pp$ is unramified, and 
	\item The result of Proposition \ref{DiscreteSeriesSmallCor} holds at $\pp$ for all $r$.
\end{enumerate}

Given this $\pp$, let $r_0$ be large enough that the result of Proposition \ref{LargePPower} holds.  Put the ideals $\{\nn_\lambda\}$ into subsequences $S_0,\, S_1,\ldots,\, S_{r_0 - 1},\, S_{r_0}$, such that, for $i < r_0$, $\nn_\lambda$ goes into $S_i$ if $\ord_\pp(\nn_\lambda) = i$.  Put $\nn_\lambda$ into $S_{r_0}$ if $\ord_\pp(\nn_\lambda) \geq r_0$.  We will show that, for any subsequence $S_i$, either $S_i$ is finite or 
\begin{equation}
\label{SubsequenceLimit}
\lim_{\substack{
		\lambda \to \infty
		\\ \nn_\lambda \in S_i}}
	\frac{|\Fam^{\leq A}(\xi,\,\chi,\,\Gamma(\nn_\lambda))|}
		{|\Fam(\xi,\,\chi,\,\Gamma(\nn_\lambda))|}
	< \epsilon.\end{equation}
Since there are finitely many subsequences, this will complete the proof.  Given $i < r_0$, for each $\nn_\lambda \in S_i$, the $\pp$-component $\Gamma_\pp$ of $\pp$ is equal to $\Gamma(\pp^i)$.

If $i = r_0$, \ref{SubsequenceLimit} follows immediately (in fact, without taking limits) by the result of \ref{LargePPower}.  So let $i < r_0$.  Let $e_{\Gamma(\pp^i),\chi_\pp}$ denote the image of $e_{\Gamma(\pp^i)}$ under the averaging map $\HH(G(F_\pp)) \to \HH(G(F_\pp),\chi_\pp)$.  Let $\mathcal{Z}$ denote the set of representations in $\Pi(G(F_\pp),\, \chi_\pp)$ that are potentially $\xi$-cohomological and that satisfy $[\QQ(\pi_\pp): \QQ] \leq A$. By \ref{Finiteness}, this set is finite; let $\one_{\mathcal{Z}}$ denote its characteristic function.  If $\nn_\lambda' = \nn_\lambda/\pp^i$, then (by Remark \ref{ComparisonRem})
$$\frac{|\Fam^{\leq A}(\xi,\,\chi,\,\Gamma(\nn_\lambda))|}
		{|\Fam(\xi,\,\chi,\,\Gamma(\nn_\lambda))|}
	\leq \frac{\wh \mu_{\Gamma(\nn_\lambda'),\xi,\chi}(\one_{\mathcal{Z}}\wh e_{\Gamma(\pp^i), \chi_\pp})}
			{\wh \mu_{\Gamma(\nn_\lambda'),\xi,\chi}(\wh e_{\Gamma(\pp^i),\chi_\pp})}$$
where $\wh \mu$ is the counting measure as defined \ref{CountingDef}.  By the Plancherel equidistribution theorem \ref{PlancherelTheorem},
$$\lim_{\lambda \to \infty}
	 \frac{\wh \mu_{\nn_\lambda',\xi,\chi}(\one_{\mathcal{Z}}\wh e_{\Gamma(\pp^i), \chi_\pp})}
			{\wh \mu_{\nn_\lambda',\xi,\chi}(\wh e_{\Gamma(\pp^i),\chi_\pp})}
	= \frac{\mupl(\one_{\mathcal{Z}}\wh e_{\Gamma(\pp^i),\chi_\pp})}
		{\mupl(\wh e_{\Gamma(\pp^i),\chi_\pp})}
	= \vol(\Gamma(\pp^i)Z/Z) \cdot \mupl(\one_{\mathcal{Z}}\wh e_{\Gamma(\pp^i),\chi_\pp}).
$$  
Since the set of representations in $\mathcal{Z}$ that are not discrete series representations is finite, its Plancherel measure is zero.  Moreover, because we have chosen $\pp$ so that the result of \ref{DiscreteSeriesSmallCor} holds, we have
$$\vol(\Gamma(\pp^i) Z/Z)
\sum_{\substack{
	\pi \text{ discrete series}
	\\ [\QQ(\pi):\QQ] \leq A}}
		\deg(\pi)\cdot \dim \pi^{\Gamma(\pp^i)}
< \epsilon .$$
Therefore, for high enough $\lambda$ such that $\nn_\lambda\in S_i$ we have
$$\frac{|\Fam^{\leq A}(\xi,\,\chi,\,\Gamma(\nn_\lambda))|}
		{|\Fam(\xi,\,\chi,\,\Gamma(\nn_\lambda))|} < \epsilon$$
finishing the proof.
\end{proof}

The remainder of the paper will be devoted to the proof of Proposition \ref{DiscreteSeriesSmallCor}.  The first step is to determine which discrete-series representations $\pi_\pp$ have small field of rationality; it turns out that if $\pp$ has large enough norm, all such representations have \emph{depth zero}.  Moreover, we will show that if $\pp$ has large enough norm, then `most' depth-zero discrete series representations (in number) have large field of rationality.  We will then have to show that the same is true when we count the discrete-series representation \emph{with multiplicity}.  In particular, we will show the following: if $\pi_i = \Sp(\pi_i',\,d)$ for $i = 1,\, 2$, where $\pi_i'$ are depth-zero supercuspidal representations, then $\deg(\pi_1) = \deg(\pi_2)$, and $\dim \pi_1^{\Gamma(\pp^r)} = \dim \pi_2^{\Gamma(\pp^r)}$.  The proof will require delving deeper into the representation theory of $\GL_n(F_\pp)$.

\section{Fields of Rationality of Tempered $\GL_n(L)$ Representations in the Tame Case}

In this section we compute explicit lower bounds on the degree of the field of rationality of tempered $\GL_n(L)$ representations with positive depth.

Throughout our proofs, we'll use the local Langlands correspondence for $\GL_n$; let $\pi$ be a $\GL_n(L)$ representation and let $\rec(\pi)$ denote the corresponding $n$-dimensional Weil-Deligne representation.  

\begin{lem} \cite[Lem VII.1.6.2]{HT01}  Let $\mathscr{L}$ denote the twisted correspondence
$$\mathscr{L}(\pi) = \rec(\pi)\otimes |\cdot|_{W_L}^{(n-1)/2}.$$
Then 
$$\mathscr{L}(\,^{\sigma}\pi) = \,^{\sigma} \mathscr{L}(\pi).$$
\end{lem}

Since the Langlands correspondence is one-to-one, we have
$$\QQ(\pi) = \QQ(\mathscr{L}(\pi)) \supseteq \QQ\left(\mathscr{L}(\pi)|_{I_L}\right) = \QQ\left(\rec(\pi)|_{I_L}\right).$$

\subsection{Supercuspidal representations in the tame case}

Throughout this subsection, let $L$ have residue characteristic $p > n$ and let $\pi$ be a supercuspidal $\GL_n(L)$ representation; then $\rec(\pi)$ is an irreducible Weil-Deligne representation (so in particular, its monodromy is trivial and the associated $W_L$ representation is irreducible).

\begin{defn} (See 2.2.1 of \cite{Moy86} or the introduction of \cite{How77}).  Let $L'/L$ be an extension of degree $n$ and let $\theta: L'^\times \to \CC^\times$ be a quasi-character.  We say $\theta$ is \emph{admissible} if 
\begin{enumerate}[(a)]
	\item $\theta$ does not factor through $N_{L'/L''}$ for any proper subextension $L''/L$, and
	\item if $\theta|_{1 + \pp_{L'}}$ factors through $N_{L'/L''}$ for some $L''$, then $L''/L'$ is unramified.
\end{enumerate}
In this case, we call the pair $(L',\,\theta)$ an \emph{admissible pair} over $L$ of degree $n$.
\end{defn}

We say $\theta_1: L_1'^\times \to \CC^\times$ and $\theta_2: L_2'^\times \to \CC^\times$ are \emph{conjugate} if there is an isomorphism $\tau: L_1\to L_2$ (over $L$) such that $\theta_1 = \theta_2 \circ \tau$.

Throughout, by abuse of notation, we will identify a character $\theta$ with the character of $W_{L'}$ associated to it by local class field theory.  The following proposition is an amalgamation of Corollary 2.1.4 and Theorem 2.2.2 of \cite{Moy86}:
\begin{prop} There is a bijection between 
$$\text{(conjugacy classes of admissible pairs over $L$ of degree $n$)}$$
and
$$\text{(irreducible $W_L$ representations of dimension $n$)}$$
given by
$$(L',\,\theta) \mapsto \rho_\theta = \Ind_{W_{L'}}^{W_L} \theta.$$
\end{prop}

Let $\pi_\theta$ be the supercuspidal $\GL_n(L)$ representation such that $\rec(\pi_\theta) = \rho_\theta$.

We say a $W_L$ representation has depth zero if it is trivial on the wild inertia subgroup of $I_L$.

\begin{lem} Let $(L',\,\theta)$ be an admissible pair over $L$.  Then the following are equivalent:
\begin{enumerate}[(a)]
	\item $\theta$ is trivial on $1 + \pp_{L'}$,
	\item $L'/L$ is unramified and $\theta$ is trivial on $1 + \pp_{L'}$
	\item $\rho_\theta$ has depth zero.
\end{enumerate}
\end{lem}
\begin{proof} Obviously (b) implies (a), and (a) implies (b) because $\theta$ is an \emph{admissible} character.

We must prove that both are equivalent to (c).  Let $P_L \leq I_L,\,  P_{L'}\leq I_{L"}$ denote the wild inertia subgroups.  Since $P_{L'}$ is the maximal pro-$p$-group of $I_{L'}$ and $1 + \pp_{L'}$ is the maximal pro-$p$-group of $\oo_{L'}^\times$, then $\theta$ is trivial on $1 + \pp_{L'}$ if and only if the associated $W_{L'}$ representation is trivial on $P_{L'}$. 

Now assume (b) holds.  Then $I_L = I_{L'}$, $P_L = P_{L'}$, and $\theta$ is trivial on $P_{L}$. The restriction of $\rho_\theta$ to $I_L$ is a direct sum of the characters conjugate to $\theta$ via $\Frob_{L}^{\ZZ}$.  Since conjugation by $\Frob_{L}$ fixes the wild inertia, then $\rho_\theta$ is trivial on the wild inertia and thus $\rho_\theta$ has depth zero.

On the other hand, assume $(L',\,\theta)$ is an admissible pair and $\rho_\theta$ is trivial on $P_L$.  By Frobenius reciprocity, $\theta|_{I_{L'}}$ is a subrepresentation of $\rho_{\theta}|_{I_{L'}}$, and so $\theta$ must be trivial on $P_{L'} \leq P_{L}$.  Thus, as a character of $L'$, $\theta$ is trivial on $1 + \pp_{L'}$, completing the proof.
\end{proof}

\begin{prop} Let $(L',\,\theta)$ be an admissible pair and let $\rho_\theta$ be the associated $W_L$ representation.
\begin{enumerate}[(i)] 
	\item $\QQ(\rho_\theta) \subseteq \QQ(\theta)$, and $[\QQ(\theta):\QQ(\rho_\theta)] \leq n$.  Moreover, the same statements hold if we replace $\theta$ by $\theta|_{\oo_{L'}^\times}$ and $\rho_\theta$ with $\rho_\theta|_{I_L}$.
	\item If $\rho_\theta$ has positive depth, then $[\QQ(\rho_\theta|_{I_L}):\QQ] \geq \frac{p-1}{n}$.
	\item If $\pi$ is supercuspidal representation of positive depth, then $[\QQ(\pi):\QQ] \geq \frac{p-1}{n}$.
\end{enumerate}
\end{prop}
\begin{proof} For (i): let $\sigma\in \Aut(\CC)$.  If $\sigma$ fixes $\theta$, then it fixes $\rho_{\theta} = \Ind \theta$.  On the other hand, assume $\sigma$ fixes $\rho_\theta$.  Then
$$\rho_\theta = \,^\sigma \rho_\theta = \,^\sigma \Ind \theta = \Ind\,^\sigma \theta$$
so $^\sigma\theta$ must be one of the $\Aut(L'/L)$-conjugates of $\theta$.  There are at most $n$ such conjugates, so the index of $\Stab(\theta)$ in $\Stab(\rho_\theta)$ is at most $n$.  The second statement follows exactly from the first upon realizing that $\rho_{\theta}|_{I_L}$ depends only on $\theta|_{\oo_{L'}^\times}$.

For (ii): if $\rho_\theta$ has positive depth then $\theta$ is nontrivial on the pro-$p$-group $1 + \pp_{L'}$, so $\theta_{\oo_{L'}^\times}$ attains the value $\zeta_p$, so the degree of its field of rationality is at least $p - 1$.  Then the statement follows from (i).

(iii) follows from (ii) because the local Langlands correspondence preserves depth (see \cite{ABPS14}).
\end{proof}



\subsection{Discrete series and tempered representations}

Continue to assume that $L$ has residue characteristic $p > n$.  Recall that the discrete series representations of $\GL_n(L)$ are of the form $\pi = \Sp(\pi',\, d)$, where $\pi'$ is a supercuspidal $\GL_m(L)$ representation and $n = md$.  On the Galois side, after ignoring the monodromy operator, $\mathscr{L}(\pi)$ is a direct sum of twists of $\mathscr{L}{\pi'}$ by powers of the absolute value character; see, for instance, page 381 of \cite{Kud91}.  In particular,
$$\QQ(\pi) = \QQ(\mathscr{L}(\pi)) = \QQ(\mathscr{L}(\pi')) = \QQ(\pi').$$

Finally, if $\pi$ is an \emph{tempered} representation, there are discrete series $\GL_{n_i}(L)$ representations $\pi_i$ for $i = 1,\ldots,\, r$ (and $n_1 + \ldots + n_r = n$) such that $\pi = \pi_1 \times \ldots \times \pi_r$.  On the Galois side, we have $\rec(\pi) = \rec(\pi_1) \oplus \ldots \oplus \rec(\pi_r)$ (as above, see page 381 of \cite{Kud91}).

\begin{prop} If $\pi$ is a tempered representation of positive depth then $[\QQ(\pi):\QQ] \geq \frac{p-1}{n}$.\end{prop}
\begin{proof} First let $\pi = \Sp(\pi',\,d)$ be a discrete series representation. Then $\QQ(\pi) \supseteq \QQ(\rec(\pi)|_{I_L}) = \QQ(\rec(\pi')|_{I_L})$. Since $\pi$ has positive depth, so does $\pi'$, and so $[\QQ(\rec(\pi')|_{I_L}): \QQ] \geq \frac{p-1}{m} \geq \frac{p-1}{n}$.

Now let $\pi = \pi_1\times \ldots \times \pi_r$ be tempered. Since $\pi$ has positive depth, so does $\pi_i$ for some $i$; assume it is $\pi_1$.  If $\sigma\in \Aut(\CC)$ fixes $\pi$, then it permutes the representations $\pi_i$; in particular, it must take $\pi_1$ to another element $\pi_i$ that is a $\GL_{n_i}(L)$ representation.  There are at most $n/n_1$ such elements, so $\Stab(\pi_1) \cap \Stab(\pi)$ is index at most $n/n_1$ in $\Stab(\pi)$.  This proves
$$[\QQ(\pi):\QQ] \geq \frac{n_1}{n} [\QQ(\pi_1):\QQ] \geq \frac{n_1}{n} \frac{p-1}{n_1} = \frac{p-1}{n}$$
since $\pi_1$ is discrete series with nonzero depth.
\end{proof}

\subsection{Counting depth-zero discrete-series representations}

In this chapter, we count the number of discrete-series $\GL_n(L)$ representations with given central character, as well as the number satisfying $[\QQ(\pi):\QQ] \leq A$.  Throughout, we will assume the residue characteristic of $L$ is $p > n$, and that the cardinality of the residue field $\FF_L$ is $q$.

We'll begin with a lemma: 
\begin{lem} If $L'/L$ is unramified and $(L',\,\theta)$ is an admissible pair, then the central character of $\pi_\theta$ is $\theta|_{L^\times}$.\end{lem}
\begin{proof} Let $[L':L] = n$ and let $\Omega: L'^\times \to \CC^\times$ be the trivial character if $n$ is odd, or the unramified sign character if $n$ is even.   It follows from Theorem 4.1.2 of \cite{Moy86} that the central character of $\pi_{\theta}$ is $(\Omega \cdot \theta)|_{L^\times}$ (to see that Moy's $\Omega$ matches ours, see 4.1.1 and 2.5.3 of that paper).  In either case, $\Omega$ is trivial on $L^\times$ and we are done.
\end{proof}

Given $A \in \ZZ_{\geq 1}$, let $f(A)$ denote the number roots of unity $\zeta$ with $[\QQ(\zeta): \QQ] \leq A$; note that $f(A)$ is finite for any $A$.  

\begin{prop} Fix a central character $\chi_0:L^\times \to \CC^\times$ that is trivial on $1 + \pp_L$, and assume $m \geq 2$.
\begin{enumerate}[(a)] 
\item Let $\beta(L,\, m,\, \chi_0)$ denote the number of depth-zero supercuspidal $\GL_m(L)$ representations $\pi'$ with $\chi_{\pi'} = \chi_0$.  Then
$$ \beta(L,\,m,\, \chi_0) \geq \frac{1}{m} \left(q^{m-1} - 1\right) $$
\item The number of depth-zero supercuspidal representations $\pi'$ with $\chi_{\pi'} = \chi_0$ and $[\QQ(\pi):\QQ] \leq A$ is bounded above by $\frac{1}{m} f(mA)$.
\end{enumerate}
\end{prop}

\begin{proof} 
Let $L'/L$ be the unique unramified extension of degree $m$. Every depth-zero supercuspidal $\GL_m(L)$ representation is of the form $\pi_{\theta}$ where $\theta: L'^\times \to \CC^\times$ is an admissible character trivial on $1 + \pp_{L'}$.  Moreover, if $\pi_\theta$ has central character $\chi_0$, then $\chi_0 = \theta\mid_{L^\times}$.  If $\varpi$ is a uniformizer of $L$, then $L'^\times$ is generated by $\oo_{L'}^\times$ and $\varpi$.  Therefore, if we insist that $\theta|_{L^\times} = \chi_{\pi_\theta} = \chi_0$, then $\theta$ is determined by its restriction $\theta_0$ to $\oo_{L'^\times}$.  An admissible character $\theta_0$ on $\oo_{L'}^\times$ trivial on $1 + \pp_{L'}$ descends to a character $\theta_0$ on $\FF_{q^m}^\times$ that does not factor through the norm map $\FF_{q^m}^\times \to \FF_{q^{x}}^\times$ for $x \neq m$; we say such a character is in \emph{general position}.  Moreover, we note that $\theta_0|_{\FF_{q}}$ must be equal to a fixed character $\chi_0$.  Therefore, it's enough to count the number $\gamma(L,\, m,\, \chi_0)$ of characters $\theta$ in general position on $\FF_{q^m}^\times \to \CC^\times$ with $\theta|_{\FF_{q}} = \chi_0$.  If $\theta$ is in general position, then its orbit under $\Gal(\FF_{q^m}/\FF_q)$ has cardinality $m$, so we have $\gamma(L,\, m,\,\chi_0) = m\cdot \beta(L,\, m,\,\chi_0)$.

We handle two separate cases: where $m > 2$ and $m = 2$.  In the $m > 2$ case, we note that the total number of characters on $\FF_{q^m}^\times$ is $q^m - 1$.  Moreover, the restriction map $\wh \FF_{q^m}^\times \to \wh \FF_q^\times$ is surjective, so the total number of characters on $\FF_{q^m}^\times$ with fixed restriction of $\FF_q^\times$ is 
$$\frac{q^m - 1}{q  - 1} = q^{m-1} + q^{m-2} + \ldots + q + 1.$$  

The number that are \emph{not} in general position is bounded above by
$$\sum_{\substack{x\mid m \\ x < m}} q^{x}.$$
If $m > 2$, $x \mid m$, and $x < m$, then $x < m - 1$, so at most $q^{m-2} + \ldots + q + 1$ characters that are not in general position, completing the proof in the $m > 2$ case.

In the $m = 2$ case, fix a central character $\chi_0 \in \wh \FF_q^\times$; we claim there are at most $2$ characters $\theta_0$ that are not in general position and such that $\theta_0|_{\FF_{q}^\times} = \chi_0$.  If $\theta$ is not in general position, then it is of the form $\theta_0' \circ N_{\FF_{q^2}/\FF_q}$ for some $\theta_0': \FF_q^\times \to \CC^\times$. On $\FF_q^\times$, $N_{\FF_{q^2}/\FF_q}$ acts as $x \mapsto x^2$, so the induced map $\wh \FF_{q}^\times \to \wh \FF_{q}^\times$ is two-to-one.  As such, given a character $\chi_0$, there are at most two characters $\theta_0'$ on $\FF_q^\times$ such that $\chi_0 = \left(\theta_0' \circ N_{\FF_q^2/\FF_q}\right)|_{\FF_{q}^\times}$, completing the proof of (a).

We now prove (b).  If $[\QQ(\pi_\theta): \QQ] \leq A$ then $[\QQ(\theta): \QQ] \leq mA$, so in particular $[\QQ(\theta_0):\QQ] \leq mA$.  The group $\FF_{q^m}^\times$ is cyclic: let it have generator $\alpha$.  If $[\QQ(\theta_0): \QQ] \leq mA$, then $\theta_0(\alpha)$ must be one of the $f(mA)$ roots of unity $\zeta$ with $[\QQ(\zeta):\QQ] \leq mA$.  Moreover, if $\theta_0$ satisfies $[\QQ(\theta_0):\QQ] \leq mA$, then so does any of its Galois conjugates, so the number of Galois orbits of characters $\theta_0$ with $[\QQ(\theta_0):\QQ]\leq A$ is at most $\frac{1}{m} f(mA)$.
\end{proof}

\begin{cor} \label{CountingCorollary} Let $p > n$, $m \geq 2$, and $n = md$.  Given $\epsilon > 0$ and $A \in \ZZ_{\geq 1}$.  
\begin{enumerate}[(i)]
	\item There is a $Q_0 > 1$ such that, for any $L$ with $|\FF_L| = q \geq Q_0$ and any character $\chi_0: L^\times \to \CC^\times$, trivial on $1 + \pp_L$ the following holds: the proportion of depth-zero $\GL_n(L)$ supercuspidal representations $\pi'$ with $\chi_{\pi'} = \chi_0$ satisfying $[\QQ(\pi'): \QQ] \leq A$ is at most $\epsilon$.
	\item With $|\FF_L| = q \geq Q_0$ and $\chi_0$ as above, the proportion of depth-zero discrete series representations $\pi = \Sp(\pi',\, d)$ with $\chi_\pi = \chi_0$ satisfying $[\QQ(\pi): \QQ] \leq A$ is at most $\epsilon$. 
\end{enumerate}
\end{cor}
\begin{proof} (i) follows directly from the above proposition. For (ii), if $\pi = \Sp(\pi',\, d)$ then $\chi_{\pi'} = \chi_{\pi'}^d$.  Fix a character $\chi_1$ with $\chi_1^d = \chi_0$.  Since $p > n > d$, each such $\chi_1$ is trivial on $1+ \pp_L$.  The proportion of $\rho$ such that $\chi_{\pi'} = \chi_d$ and $[\QQ(\pi'):\QQ] \leq A$ is at most $\epsilon$.  Since $\QQ(\pi) = \QQ(\pi')$ this completes the proof.
\end{proof}

We will also need the following lemma:
\begin{lem} Fix an unramified central character $\chi_0$ of $L^\times$, where $L$ has residue characteristic $p > n$.  The number of Steinberg representations $\St(\chi)$ of $\GL_n(L)$ with central character $\chi_0$ is bounded above by $n^2$.\end{lem}
\begin{proof} The central character of $\St(\chi)$ is $\chi^n = \chi_0$, so $\chi(\varpi_L)$ is one of the $n$ roots of $x^n = \chi_0(\varpi)$.  Moreover $\chi^n|_{\oo_{L}^\times}$ is trivial, so $\chi$ must be trivial on $1 + \pp_L$: otherwise it attains the value $\zeta_p$ and $\zeta_p^n \neq 1$ since $p > n$.  Therefore, $\chi|_{\oo_L^\times}$ factors through a cyclic group with generator $\alpha$, and we must have $\chi(\alpha)^n = 1$, so $\chi|_{\oo_L^\times}$ is one of at most $n$ characters. This completes the proof, since $L^\times = \oo_L^\times \times \varpi_L^\ZZ$.
\end{proof}

\section{Computing multiplicities of depth-zero discrete series representations}

In this section, we will complete the proof of \ref{DiscreteSeriesSmallCor}.  Fix $A\geq 1,\,\epsilon > 0$ and let $p > nA$, so that if $\pp\mid p$ and $\pi$ satisfies $[\QQ(\pi):\QQ] \leq A$, then $\pi$ has depth zero.  In view of this fact, it is enough to prove:

\begin{prop} \label{DiscreteSeriesSmall} Fix $\epsilon > 0,\,A \geq 1$ and fix $d\mid n$.  There is a $P_0 > nA \in \ZZ$ such that, for any $p > P_0$, the following holds: For every $L$ of residue characteristic $p > P_0$, consider $\Gamma(\pp^r) \leq \GL_n(L)$.  For every $r$, and for every unramified character $\chi: L^\times \to \CC^\times$, we have
\begin{equation}\label{DSS1}
\sum_{\substack{
	\\ \pi \in \Pi(\GL_n(L),\,\chi)
	\\ \pi = \Sp(\pi',\,d)
	\\ [\QQ(\pi):\QQ] \leq A
	\\ \dep(\pi) = 0 }}
		\deg(\pi)\cdot \dim \pi^{\Gamma(\pp^r)}
\leq \epsilon \cdot \vol(\Gamma(\pp^r) Z/Z)^{-1}.\end{equation}
\end{prop}

If we fix $n$, there are finitely many $d \mid n$, and if $p > nA$ then all representations $\pi$ with $[\QQ(\pi):\QQ] \leq A$ have depth zero, so it is clear that this implies Proposition \ref{DiscreteSeriesSmallCor}

We will prove Proposition \ref{DiscreteSeriesSmall} in two different cases: the case where $d = n$ (so $\pi = \Sp(\pi',\,n)$ is a \emph{standard} Steinberg representation), and separately in the case $d < n$.

\subsection{Multiplicities of Steinberg representations}

In this section, we prove Proposition \ref{DiscreteSeriesSmall} in the case $d = n$.  As always, $L$ has prime ideal $\pp$, residue characteristic $p$, and $|\FF_L| = q$.  Recall that we have chosen a Haar measure such that $\K_\pp Z /Z$ has measure one, so
$$\vol(\Gamma(\pp^r)Z/Z)^{-1} = q^{(n^2 - 1)(r - 1)}|\PGL_2(\FF_q)| > c\cdot q^{(n^2 - 1)r}$$
where $c > 0$ is a constant independent of $q,\,n$, and $r$.

When $\pi$ is a Steinberg representation, it follows from (2.2.2') of \cite{CMS90} that
$$\deg(\pi) = \frac{1}{n} \prod_{k = 1}^{n-1} (q^k - 1) \leq \frac{1}{n} q^{n(n-1)/2}$$

We need to compute an upper bound on $\dim \pi^{\Gamma(\pp^r)}$.

\begin{lem} There is a $C > 0$ such that, for all $L$ with $|\FF_L| = q$ and all characters $\chi_{0}: L^\times \to \CC^\times$ of conductor at most $1$, we have
$$\dim \St(\chi_0)^{\Gamma(\pp^r)} \leq C q^{rn(n-1)/2}.$$
\end{lem}
\begin{proof} Since $\chi$ has conductor $1$ and $r \geq 1$, it is clear that $\St(\chi)$ and $\St(\one)$ have the same dimension of $\Gamma(\pp^r)$-fixed vectors.  Let $I$ denote the (unnormalized) induction $\Ind_{B}^G \one$, where $\one$ is the trivial representation of the Borel subgroup $B$.  Then $\dim \St^{\Gamma(\pp^r)} \leq \dim I^{\Gamma}$ since $\St$ is a quotient of $I$ (for admissible representations, $\dim \pi^{\Gamma} = \tr \pi(e_\Gamma)$, so the function $\pi \mapsto \dim \pi^{\Gamma}$ is additive in exact sequences).

Let $V$ be the space of $I$.  Using Mackey's Theorem, we have
\begin{align*} V^{\Gamma} & \cong \bigoplus_{g\in B \bs G / \Gamma} \CC^{B \cap g \Gamma g^{-1}}
\\ & = \bigoplus_{g\in B \bs G / \Gamma} \CC
\end{align*}
so it suffices to find the cardinality of $B \bs G /\Gamma$.  Since $B\K = G$ and $\Gamma$ is normal in $\K$, we have
$$B \bs G / \Gamma = (B \cap \K) \bs \K / \Gamma = ((B\cap K) \Gamma) \bs \K.$$

The group $(B \cap \K)\cdot \Gamma(\pp^r)$ consists of those matrices in $\K$ such that the elements below the diagonal are in $\pp^r$.  As such, there is a fixed $C$ such that $((B \cap \K)\cdot \Gamma(\pp^r)) \bs \K \leq C q^{r(n-1)n/2}$, completing the proof.
\end{proof}

Moreover, there are at most $n^2$ Steinberg representations with given central character, this proves
$$\sum_{
		\St(\chi_0) \in \Pi(\GL_n(L),\,\chi} 
	\deg(\St(\chi_0)) \dim \St(\chi_0)^{\Gamma(\pp^r)} 
	\leq nCq^{(r+1)n(n-1)/2}.$$

This completes the proof in the Steinberg case.

\subsection{Multiplicities of generalized Steinberg representations}

In this section, we prove Proposition \ref{DiscreteSeriesSmall} in the case $d \neq n$.  We will prove the following lemma:
\begin{lem}\label{equalities} Assume $L$ has residue characteristic $p > 2n$ and fix $d$ with $n = md$.  Let $\pi_i'$ be a depth-zero supercuspidal $\GL_m(L)$ representation and let $\pi_i = \Sp(\pi_i',\,d)$ for $i = 1,\,2$.  Then
$$\deg(\pi_1) = \deg(\pi_2)$$
and
$$\dim \pi_1^{\Gamma(\pp^r)} = \dim \pi_2^{\Gamma(\pp^r)}$$
for all $r$.
\end{lem}

\begin{proof}[Proof of \ref{DiscreteSeriesSmall} assuming \ref{equalities}]
Let $e_\Gamma \in \HH(\GL_n(L))$ be the idempotent element corresponding to $\Gamma(\pp^r)$ and let $e_{\Gamma,\,\chi} \in \HH(\GL_n(L),\,\chi)$ be its image under the averaging map.  Then
$$\vol(\Gamma(\pp^r)Z/Z)^{-1}  = e_{\Gamma,\,\chi}(1) = \mupl_{\chi}(\wh e_{\Gamma,\chi})$$
which, in particular, is greater than the integral of $\wh e_{\Gamma,\,\chi}$ over the subset of $\Pi(\GL_n(L),\,\chi)$ consisting of the representations $\pi = \Sp(\pi',\, d)$.  Therefore we have
$$\vol(\Gamma(\pp^r) Z/Z)^{-1} \geq  
\sum_{\substack{
	\\ \pi \in \Pi(\GL_n(L),\,\chi)
	\\ \pi = \Sp(\pi',\,d)
	\\ \dep(\pi) = 0}}
		\deg(\pi)\cdot \dim \pi^{\Gamma(\pp^r)}$$
and so
$$\vol(\Gamma(\pp^r)Z/Z) \sum_{\substack{
	\\ \pi \in \Pi(\GL_n(L),\,\chi)
	\\ \pi = \Sp(\pi',\,d)
	\\ [\QQ(\pi):\QQ] \leq A
	\\ \dep(\pi) = 0 }}
		\deg(\pi)\cdot \dim \pi^{\Gamma(\pp^r)}$$
$$\leq 
\left(\sum_{\substack{
	\\ \pi \in \Pi(\GL_n(L),\,\chi)
	\\ \pi = \Sp(\pi',\,d)
	\\ [\QQ(\pi):\QQ] \leq A
	\\ \dep(\pi) = 0}}
		\deg(\pi)\cdot \dim \pi^{\Gamma(\pp^r)}\right)
\bigg/
\left(\sum_{\substack{
	\\ \pi \in \Pi(\GL_n(L),\,\chi)
	\\ \pi = \Sp(\pi',\,d)
	\\ \dep(\pi) = 0}}
		\deg(\pi)\cdot \dim \pi^{\Gamma(\pp^r)}\right).
$$

By Lemma \ref{equalities}, the terms of the sums have the same value (in the numerator and the denominator), so by Corollary \ref{CountingCorollary}, the quotient is at most $\epsilon$.
\end{proof}

We'll first prove that the formal degrees of two such representations are equal:

\begin{lem} Let $\pi'$ be a depth-zero supercuspidal $\GL_m(L)$ representation and let $\pi = \Sp(\pi',\,d)$.  Let $\St_{m}$ denote the Steinberg representation of $\GL_m(L)$. Then
$$\deg(\pi) =
	\deg(\St_m)^d
	\cdot \frac{m^d}{d}
	\cdot \frac{(q^m - 1)^d}
		{q^{md} - 1}
	\cdot \frac{|\GL_{dm}(\FF_q)|}
		{|\GL_{m}(\FF_q)|^d}
	\cdot q^{-m^2 \frac{d^2 - d}{2}}
$$
\end{lem}

\begin{proof} First, if $\pi'$ is a depth-zero supercuspidal $\GL_m(L)$ representation, we may compute the formal degree $\deg(\pi')$ using Theorem 2.2.8 of \cite{CMS90}. In this case, $\pi'$ is associated to the admissible pair $(L_m,\, \eta)$ where $L_m/L$ is the unramified extension of degree $m$, and $\eta: L_m^\times \to \CC^\times$ is trivial on $1 + \pp_{L_m}^\times$.  Using the notation of that Theorem, we compute $\alpha(\theta) = m-1,\, f = m,\, e = 1$, so 
$$\deg(\pi') = m\cdot \deg(\St).$$

We now use this to compute $\deg \pi$; in view of Theorem 6.3 of \cite{AP05} we have
$$\frac{\deg(\pi)}{\deg(\pi')^d} = 
	\frac{m^{d-1}}{r^{d-1}d} 
	\cdot q^{\frac{d^2 - d}{2} (f(\pi'^{\vee} \times \pi') + r - 2m^2)}
	\cdot \frac{(q^r - 1)^d}{q^{dr} - 1}
	\cdot \frac{|\GL_{dm}(\FF_q)|}{|\GL_{m}(\FF_q)|^d}.$$
Here $r$ is the number of unramified characters $\chi: L^\times \to \CC^\times$ such that $\pi' \otimes (\chi \circ \det) \cong \pi'$; (or the \emph{torsion number} of $\pi'$), and $f(\pi' \times \pi'^{\vee})$ is the conductor of the pair $\pi' \times \pi''$.

We first prove $r = m$.  Let $\chi: L^\times \to \CC^\times$ be an unramified character.  We note that $\pi' = \pi_{\theta}$ where $\theta$ is a depth-zero character of $L_m^\times$ and $L_m$ is the unramified extension of $L$ of degree $m$.  Moreover $(\chi\circ \det) \otimes \pi_{\theta} \cong \pi_{(\chi \circ N_{L_m/L}) \cdot \theta}$.  If $\chi(\varpi_L)^m = 1$, then $\chi \circ N_{L_m/L} \equiv 1$, so the torsion number is at least $m$. To see it is exactly $m$, the central character of $\pi\otimes (\chi\circ \det)$ is $\chi_\pi \cdot \chi^m$. Thus, we must have $\chi_m(\varpi_L)^m = 1$, and there are $m$ such unramified characters.

To prove that $f(\pi' \times \pi'^{\vee}) = m^2 - m$, we use the local Langlands correspondence.  Let $\rec(\pi') = (\rho,\,V,\,N)$; since $\pi'$ is supercuspidal, $\rec(\pi)$ is irreducible, and since the monodromy is trivial and so we need only consider $(\rho,\, V)$ as a representation of the Weil group $W_L$.   

Then $f(\pi' \times \pi'^{\vee}) = f(\rho \otimes \rho^\vee)$.  Since the Langlands correspondence preserves depth, $\rho$ is trivial on the wild inertia subgroup $P_L$.  Since $I_L/P_L$ is abelian, $\rho|_{I_L}$ decomposes as a direct sum of abelian characters $\theta_1,\, \ldots,\, \theta_m$.  Because $\rho$ is irreducible and Frobenius-semisimple and since the Frobenius element permutes the spaces of the characters $\theta_i$, the characters $\theta_i$ are pairwise distinct.  Therefore $(\rho \otimes \rho^\vee)$ is trivial on the ramification group $I_L^1$, and the subspace of $I_L$-fixed vectors has dimension $m$ (corresponding to the spaces of $\theta_i \otimes \overline \theta_i$).  From the definition (when the monodromy is trivial), we discern 
$$f(\rho \otimes \rho^{\vee}) = \int_0^\infty \codim(V\otimes V^{\vee})^{I_L^j}\,dj = m^2 - m$$
as desired.

Now the proof is complete once we plug $f = m^2 - m,\, r = m$ into Aubert and Plymen's formula.
\end{proof}

To prove that $\dim \pi_1^{\Gamma(\pp^r)} = \dim \pi_2^{\Gamma(\pp^r)}$, we will prove an even stronger result.  Let $\Theta_{\pi}$ denote the Harish-Chandra character of $\pi$; this is a conjugation-invariant function, defined up to a set of measure zero, so that for any $e\in \HH(\GL_n(L))$, we have 
$$\tr \pi(e) = \int_{\GL_n(L)} \Theta_{\pi}(g) \cdot e(g)\,dg.$$
We will show the following:

\begin{lem} \label{CharacterEquality} As above, let $\pi_i'$ be a depth-zero supercuspidal $\GL_m(L)$ representation and let $\pi_i = \Sp(\pi_i,\, d)$ for $i = 1,\,2$.  Then $\Theta_{\pi_1}$ and $\Theta_{\pi_2}$ can be chosen to be equal on $\Gamma(\pp)$.
\end{lem}

This implies $\dim \pi_1^{\Gamma(\pp^r)} = \dim \pi_2^{\Gamma(\pp^r)}$, since $\dim \pi_i^{\Gamma(\pp^r)} = \tr \pi_i(e_{\Gamma(\pp^r)})$, and $e_{\Gamma(\pp^r)}$ is supported on $\Gamma(\pp^r)\leq \Gamma(\pp)$.

We will use a simplified version of the character expansion formula of \cite{Mur03} that is appropriate for our situation.  We begin with some preliminaries.  Let $G = \GL_n(L),\, \fg = \Lie(G)$.  Let $\fg_{\reg},\, G_{\reg}$ denote the set of regular semisimple elements in $\fg,\, G$ respectively.  Fix once and for all an additive character $\psi$ on $L$, and let $\linf{\cdot}{\cdot}$ denote the bilinear form on $\fg$ given by $\linf{X}{Y} = \psi(\tr(XY))$.  Assume further that the Haar measure $dX$ on $\fg$ is self-dual with respect to $\linf{\cdot}{\cdot}$.  Given an $\Ad(G)$ orbit $\OO\subset \fg$, let $\mu_{\OO}(f)$ denote the integral of $f$ over the orbit $\OO$.  If $\wh f$ is the Fourier transform of $f$, we define the distribution $\wh \mu_\OO$ via $\wh \mu_\OO(f) = \mu_\OO(\wh f)$.  The distribution $\wh \mu_\OO$ is representable by a locally integrable function on $\fg$, which by abuse of notation we also call $\wh \mu_\OO$.  In particular
$$\wh \mu_\OO(f) = \mu_\OO(\wh f) = \int_\fg f(X)\wh \mu_\OO(X)\,dX$$
for any $f\in C_c^\infty(\fg)$.

Given a semisimple $s\in \fg$, let $\Gamma_G(s)$ denote the set of orbits $\OO\subset \fg$ that contain $s$ in their closure.

In \cite{Mur03}, Murnaghan gives a description of the character of a depth-zero discrete series representation $\pi = \Sp(
\pi',\, d)$, where $n = dm$.  Let $L'/L$ be the unique unramified extension of degree $n$, and choose a basis $\{\xi_1,\ldots\, \xi_n\}$ of $L'/L$ such that $\xi_i\in \oo_{L'}$ and such that their reductions modulo $\pp_{L'}$ form a basis of $\FF_{L'}$ over $\FF_L$.  This choice of basis defines an embedding $L' \into M_n(L) = \fg$.  For any $m\mid n$, consider the unramified subextension $L_m/L$ of degree $m$; then we get an embedding $L_m\into \fg$.  Choose an element $s_m \in \oo_{L_m} - \pp_{L_m}$ whose reduction modulo $\pp_{L}$ generates $\FF_{L_m}$ over $\FF_L$ and identify it with its image in $\fg$.

\begin{thm} \label{Murnaghan} \begin{enumerate}[(a)]
\item Let $\pi'$ be a depth-zero supercuspidal $\GL_m(L)$ representation and let $\pi = \Sp(\pi',\, d)$.  With $s_m$ as above, there are constants $c_{\OO}$ for $\OO\in\Omega(s_m)$ such that 
$$\Theta_\pi(1 + X) = \sum_{\OO\in \Omega_G(s_m)} c_{\OO,\fg}\cdot \wh \mu_{\OO}(X)$$
for all $X\in \pp \cdot M_n(\oo_L)$.

\item Moreover, the constants $c_{\OO}$ can be chosen as follows: let $H = C_{G}(s_m) \cong \GL_{d}(L_m)$. Then there is bijection $\Omega_{H}(0) \leftrightarrow \Omega_G(s_m)$ given by 
\begin{equation} 
	\OO_{H} \mapsto \Ad(G)\cdot (s_m + \OO_{H}). 
	\label{OrbitBijection}
\end{equation}
Let $\St_H$ denote the Steinberg represenation of $H$; then the character of $\St_H$ has a $0$-asymptotic extension:
$$\Theta_{\St_H}(1 + X) =\sum_{\OO_H \in \Omega_{H}(0)} c_{\OO_H} \cdot \wh \mu_{\OO_H}(X).$$

There is a constant $\lambda$, depending only on a choice of Haar measures on $G,\, H$, such that 
$$c_{\OO} = \lambda c_{\OO_H}$$
if $\OO$ and $\OO_H$ correspond under \ref{OrbitBijection}, assuming measures on $\OO$ and $\OO_H$ are chosen compatibly.
\end{enumerate}
\end{thm}

Lemma \ref{CharacterEquality} follows as a corollary, since the above description of the character $\Theta_\pi$ is true for all $\pi = \Sp(\pi',\,d)$ where $\pi'$ is a depth-zero supercuspidal $\GL_m(L)$ representation.

\begin{proof}[Proof of \ref{Murnaghan}] We briefly recall some facts about the theory of types in our situation.  Let $\pi = \Sp(\pi',\,d)$ where $\pi'$ is a depth-zero supercuspidal representation of $\GL_m(L)$.  Then $\pi'$ is compactly induced from a depth-zero $Z\cdot \GL_m(\oo_L)$ representation $\kappa$; let $\kappa_0$ denote its restrictions $\GL_m(\oo_L)$.

Let $P_m \leq \GL_n(\FF_L)$ be the standard parabolic subgroup whose Levi component $M_m$ is isomorphic to $\GL_m(\FF_L)^{\oplus d}$.  Let $B_m \leq \GL_n(\oo_L)$ denote the preimage of $P_m$ under $\GL_n(\oo_L) \twoheadrightarrow \GL_n(\FF_L)$; this is a parahoric subgroup of $\GL_n(L)$.  Moreover, there is a surjection $B_m \to \GL_m(\oo_L)^d$.
Consider $\kappa_0^{\otimes d}$ as a representation of $\GL_m(\oo_L)^{\oplus d}$.  Let $\tau$ be the pullback of $\kappa_0^{\otimes d}$ to $B_m$.  Then $\pi|_{B_m}$ contains an subrepresentation isomorphic to $\tau$.  In this situation we say $\pi$ contains the \emph{totally pure refined minimal $K$-type} $(B_m,\, \tau)$.

Therefore, \cite[Theorem 14.1]{Mur03} applies.  The content of (2) is that $\Theta_\pi$ has an asymptotic expansion about some element $s_{\tau,h} \in \fg$, defined in (1) of that Theorem and that the expansion is valid for $X\in \pp\cdot M_n(\oo_L)$ since $\pi$ has depth zero.  Moreover, if $\pi = \Sp(\pi',\,d)$, then $s_{\tau, h} = s_m$; this is checked in \cite[Lemma 10.9]{Mur03}.

Moreover, we can compute the constants $c_\OO(\pi)$ using (3).  Let $H = \GL_d(L_m)$ and let $B_H\leq H$ denote an Iwahori subgroup.  Then there is a Hecke algebra isomorphism
$$\HH(G\,/\,/\, B_m,\,\tau^*) \cong \HH(H\,/\,/\, B_H,\,\one)$$
(see, for instance, Theorem 10.2 of \cite{Mur03} or Theorem 1.2 of \cite{How85}).
This gives a bijection between the sets of irreducible $G$ representations containing $(B_m,\, \tau)$ and irreducible $H$ representations having a nonzero Iwahori-fixed vector.  Moreover, because the isomorphism respects the $L^2$ structures on $\HH$, it takes discrete-series representations to discrete-series representations.  In particular, $\pi$ corresponds to a Steinberg representation of $H$ under this isomorphism.

Finally, we must show that all Steinberg representations of $H$ having an Iwahori-fixed vector have the same asymptotic expansion.  All such representations are of the form $\St(\chi_m)$ for some character $\chi_m:L_m^\times \to \CC^\times$, that is trivial on $1 + \pp_{L_m}$, and $\St(\chi_m) \cong \chi_m \otimes \St(\one)$.  Since $\chi_m$ is trivial on $1 + \pp_{L_m}$, the characters of $\St(\chi_m)$ and $\St(\one)$ are the same on $\Gamma(\pp_{L_m})$, completing the proof.
\end{proof}

\section{Other classical groups: A roadmap}

In this section, we give a brief outline of a potential proof for other classical groups (for instance, special orthogonal or symplectic groups).  The primary difficulty is that the depth-zero discrete series representations are not as-well understood as those of $\GL_n$ and in particular we are unaware of a proof an analog of \ref{DiscreteSeriesSmallCor}.

Throughout, this section, $F$ will denote a totally real field and $G/F$ a special orthogonal or symplectic group with center $Z$, $\xi$ will denote an irreducible finite dimensional algebraic $G(F_\infty)$ representation and $\chi$ an automorphic character of $Z(\AAA_F)$ of conductor $\ff$ with $\chi_{\infty} = \chi_{\xi}$.  We continue to assume $\xi$ has regular highest weight.  As before, let $\nn$ denote an ideal of $F$ and $\Gamma(\nn)$ the principal congruence subgroup of level $\nn$.

Let $\Fam(\xi,\,\chi,\, \Gamma(\nn))$ denote the family of cuspidal automorphic $\xi$-cohomological representations $\pi$, counted with multiplicity $m_{\cusp}(\pi) \dim(\pi^\infty)^{\Gamma(\nn)}$.

Our conjecture is the following:
\begin{conj} \label{Conjecture} Given $F,\, G,\, \xi,\, \chi$ and $\ff$ as above, let $\{\nn_\lambda\}$ be a sequence of ideals divisible by $\ff$ such that $N(\nn_\lambda) \to \infty$ as $\lambda \to \infty$.  Then
$$\lim_{\lambda\to\infty} 
	\frac{|\Fam^{\leq A}(\xi,\,\chi,\, \Gamma(\nn_\lambda))|}
		{|\Fam(\xi,\,\chi,\, \Gamma(\nn_\lambda))|} = 0.$$
\end{conj}

We begin with a proposition:
\begin{prop} Let $G$ be a special orthogonal or symplectic group over a totally real field $F$.  Assume the appropriate analog of \ref{DiscreteSeriesSmallCor} holds for $G$.  Then Conjecture \ref{Conjecture} holds.
\end{prop}
\begin{proof} We may trace through the steps of the proof of Theorem \ref{MainThm}.  First, we replicate the proof of \ref{PlancherelTheorem} (following, once again, the argument in \cite{ST12}) to prove Plancherel equidisitribution for the sequence of measures $\wh\mu^{\cusp}_{\Gamma(\nn_\lambda),\, \chi_S,\, \xi}$. Moreover, since the results in \cite{ST13} hold for special orthogonal and symplectic groups as well, we can prove the analogs of Propositions \ref{Finiteness} and \ref{LargePPower}. 

Now, if we fix $\epsilon$ and $A$, we may pick $\pp$ with large enough norm that the result of \ref{DiscreteSeriesSmallCor} holds and an $r_0$ such that, for any $\nn$ with $\pp^{r_0}\mid \nn$ we have
$$\frac{|\Fam^{\leq A}(\xi,\,\chi,\, \Gamma(\nn))|}
		{|\Fam(\xi,\,\chi,\,\Gamma(\nn))|} < \epsilon.$$
		
Finally, the result follows as in the proof in Subsection 4.2 by breaking the sequence $\{\nn_\lambda\}$ into subsequences and handling each subsequence separately using Plancherel equidistribution.
\end{proof}

\subsection{Fields of rationality of supercuspidal representations}

In this subsection we briefly discuss a potential approach to computing the field of rationality of supercuspidal representations.  We caution the reader that we have not yet checked all the details:

\begin{conj} Let $G/F$ be a connected reductive group, and let $|W|$ denote the cardinality of the Weyl group of a maximal torus in $G \times \bar F$.  Let $\pp$ be a prime such that $F_{\pp}$ has sufficiently high residue characteristic $p$ (depending only on $G$).  If $\pi$ is a supercuspidal representation of $G(F_{\pp})$ of positive depth, then
$$[\QQ(\pi):\QQ] \geq \frac{p-1}{|W|}.$$
\end{conj}

The proof uses analogs by Kim, Murnaghan, and Yu to the results of Howe and Moy for $\GL_n$.  In particular, in \cite{Yu01}, Yu constructs a tame supercuspidal representation associated to a cuspidal datum $\Phi = (\overrightarrow{G}, \, y,\, \overrightarrow{r},\, \rho,\, \overrightarrow{\phi})$ (we refer the reader to Section 3 of \cite{Yu01} for precise definitions).  In \cite{Kim07}, Kim proves that when the residue characteristic is large enough, this construction is exhaustive. 

If $\sigma \in \Aut(\CC)$ and $\Phi = (\overrightarrow{G},\, y,\, \overrightarrow{r},\, \rho,\, \overrightarrow{\phi})$, let $^\sigma \Phi = (\overrightarrow{G},\, y,\, \overrightarrow{r}, \,^\sigma \rho,\, ^\sigma\overrightarrow{\phi})$, where $^\sigma \overrightarrow{\phi} = (^\sigma \phi_0,\ldots,\, ^\sigma \phi_d)$.  It needs to be checked that Yu's construction $\Phi\mapsto \pi$ is equivariant under the action of $\Aut(\CC)$.

Let $\Phi$ be a cuspidal datum and let 
$$\rho' = \rho \otimes \prod_{i = 0}^d (\phi^i)^{-1}|_{K^0}.$$
It follows from Theorem 6.7 of \cite{HM08} (up to a hypothesis on the cuspidal datum) that $\pi_{\Phi_0} = \pi_{\Phi_1}$ if and only if there is a $g\in G(L)$ with $\Ad(g) K_0^0 = K_1^0$, and such that $\Ad(g)$ takes $\rho'_0$ to $\rho'_1$.

Let $W(G)$ be the Weyl group of a maximal torus in $G$ (after base change to the algebraic closure): one can prove that $[N_{G(L)}(K^0): K^0] \leq |W(G)|$.  Now if $\sigma \in \Aut(\CC)$ fixes $\pi_{\Phi}$, then $^\sigma \rho'$ must be equal to $\Ad(x) \rho'$ for some $x\in N_{G(L)}(K^0)/K^0$.  Therefore, we have $\QQ(\pi_\Phi) \subseteq \QQ(\rho'_\Phi)$, and $[\QQ(\rho'_\Phi): \QQ(\pi_\Phi)] \leq |W(G)|$.

\end{document}